\documentclass[a4paper]{article} 

\usepackage{amsmath, amssymb, array, booktabs, braket, mathrsfs}
\usepackage{type1cm}
\usepackage{amsthm}
\usepackage[dvipdfm]{hyperref}

\theoremstyle{definition}
\newtheorem{dfn}{Definition}[section]
\newtheorem{thm}[dfn]{Theorem}
\newtheorem{lem}[dfn]{Lemma}
\newtheorem{prp}[dfn]{Propostion}
\newtheorem{cor}[dfn]{Corollary}

\newtheorem{rem}[dfn]{Remark}
\newtheorem*{out}{Outline}
\newtheorem*{ack}{Acknowledgement}
\newtheorem{cla}{Claim}
\newtheorem*{pf}{Proof of Theorem \ref{yandong}}
\newtheorem*{pf2}{Proof of Corollary \ref{yandong2}}
\newtheorem*{pf3}{Proof of Corollary \ref{yandong3}}

\begin{document}

\title{Stable lattices in modular Galois representations and Hida deformation}
\author{Dong Yan}
\date{}
\maketitle
\begin{abstract}
In this paper, we discuss the variation of the numbers of the isomorphic classes of stable lattices when the weight and the level vary in a Hida deformation by using the Kubota-Leopoldt $p$-adic $L$-function. Then in Corollary \ref{yandong3}, we give a sufficient condition for the numbers of the isomorphic classes of stable lattices in Hida deformation to be infinite.
\end{abstract}
\tableofcontents

\section{Introduction}

Fix a prime $p \geq 3$. We denote by $\mathbf{Q}\left(\mu_{p^{\infty}} \right)$ the extension of the field of rational numbers $\mathbf{Q}$ obtained by adjoining all $p$-power roots of unity. We fix a complex embedding $\overline{\mathbf{Q}} \hookrightarrow \mathbf{C}$ and a $p$-adic embedding $\overline{\mathbf{Q}} \hookrightarrow \overline{\mathbf{Q}}_p$ of an algebraic closure $\overline{\mathbf{Q}}$ throughout the paper, where $\mathbf{C}$ is the field of complex numbers and $\overline{\mathbf{Q}}_p$ an algebraic closure of the field $\mathbf{Q}_p$ of $p$-adic numbers. We fix a topological generator $u$ of $1+p\mathbf{Z}_p$ throughout the paper. We denote by $\mathbf{Q}_{\infty}$ the cyclotomic $\mathbf{Z}_p$-extension of $\mathbf{Q}$. Let $$\chi_{\mathrm{cyc}} : \mathrm{Gal}\left(\overline{\mathbf{Q}}/\mathbf{Q} \right) \twoheadrightarrow \mathrm{Gal}\left(\mathbf{Q}\left(\mu_{p^{\infty}} \right)/\mathbf{Q} \right) \stackrel{\sim}{\rightarrow} \mathbf{Z}_p^{\times}$$ be the $p$-adic cyclotomic character. Thus $\chi_{\mathrm{cyc}}$ is decomposed into the product $\chi_{\mathrm{cyc}}=\kappa_{\mathrm{cyc}}\omega$ where $$\kappa_{\mathrm{cyc}} : \mathrm{Gal}\left(\mathbf{Q}_{\infty}/\mathbf{Q} \right) \stackrel{\sim}{\rightarrow} \mathrm{Gal}\left(\mathbf{Q}\left(\mu_{p^{\infty}} \right)/\mathbf{Q}\left(\mu_p \right) \right) \stackrel{\chi_{\mathrm{cyc}}}{\rightarrow} 1+p\mathbf{Z}_p$$ is the canonical character and 
$$\omega : \mathrm{Gal}\left(\mathbf{Q}\left(\mu_p \right)/\mathbf{Q} \right) \stackrel{\sim}{\rightarrow} \mathrm{Gal}\left(\mathbf{Q}\left(\mu_{p^{\infty}} \right)/\mathbf{Q}_{\infty} \right) \stackrel{\chi_{\mathrm{cyc}}}{\rightarrow} \mu_{p-1}$$ the Teichm$\ddot{\mathrm{u}}$ller character. Let $\mathcal{O} \subset \overline{\mathbf{Q}}_p$ be a commutative ring which is finite flat over $\mathbf{Z}_p$ and let $\psi$ be a Dirichlet character modulo $M$. We denote by $S_k\left(\Gamma_0\left(M \right), \psi, \mathcal{O} \right)$ the space of cusp forms of weight $k$, level $M$, Neben character $\psi$ and Fourier coefficients in $\mathcal{O}$. We also denote by the same symbol $\psi$ the corresponding character of $\mathrm{Gal}\left(\mathbf{Q}(\mu_M)/\mathbf{Q} \right) \cong \left(\mathbf{Z}/M \right)^{\times}$. For a group $\Delta_{M}$ which is isomorphic to $\mathrm{Gal}\left(\mathbf{Q}(\mu_M)/\mathbf{Q} \right)$, a $\mathbf{Z}_p$-module $\mathcal{M}$ which has a $\mathbf{Z}_p$-linear action of $\Delta_{M}$ and a character $\varepsilon$ of $\Delta_{M}$, we denote by $\mathcal{M}^{\varepsilon}=\mathcal{M} \displaystyle\otimes_{\mathbf{Z}_p[\Delta_M]} \mathbf{Z}_p[\varepsilon]$.

In 1976, Ribet \cite{Ri76} proved the converse of Herbrand's theorem as follows:
\begin{thm}[Ribet]\label{1.1}

Let $k$ be an even integer satisfying $2 \leq k \leq p-3$ and $B_k$ the $k$-th Bernoulli number. We denote by $\mathrm{Cl}(\mathbf{Q}(\mu_p))[p^{\infty}]$ the $p$-part of the ideal class group of $\mathbf{Q}(\mu_p)$ on which the Galois group $\mathrm{Gal}(\overline{\mathbf{Q}}/\mathbf{Q})$ acts by functoriality. Suppose $p$ divides $B_k$. Then $\mathrm{Cl}(\mathbf{Q}(\mu_p))[p^{\infty}]^{\omega^{1-k}} \neq 0.$

\end{thm}

The method of Ribet's proof is to construct a normalized Hecke eigen cusp form $f=\displaystyle\sum_{n=1}^{\infty}a(n, f)q^n \in S_2(\Gamma_0(p), \chi)$ which is congruent to Eisenstein series by the condition $p$ divides $B_k$. Then by using the Galois representation $\rho_f$ attached to $f$ due to Deligne and Shimura, Ribet constructed an unramified $p$-extension of $\mathbf{Q}(\mu_p)$ by using a canonical stable lattice (see Proposition \ref{1.2} below) of $\rho_f$. By extending Ribet's method, Mazur-Wiles \cite{MW} and Wiles \cite{Wi90} proved the Iwasawa main conjecture for $\mathbf{Q}$ and for totally real fields.

The key lemma Ribet used, which is called ``Ribet's lemma", is the following proposition:

\begin{prp}[Ribet's lemma]\label{1.2}

Let $\left(\mathcal{O}, \varpi, \mathcal{O}/\left(\varpi \right) \right)$ be the ring of integers of a finite extension of $\mathbf{Q}_p$ where $\varpi$ is an uniformizer of $\mathcal{O}$. Let $K=\mathrm{Frac}(\mathcal{O})$ be the fraction field of $\mathcal{O}$ and $V$ a $2$-dimensional $K$-vector space. For a given $p$-adic representation $$\rho : G \rightarrow \mathrm{Aut}_K(V)$$ of a compact group $G$, let $\bar{\rho}^{\mathrm{ss}}$ be the semi-simplification of the mod $\varpi$ representation (see Section 2.1 below).
Suppose $\rho$ is irreducible and $\bar{\rho}^{\mathrm{ss}} \cong \psi_1 \oplus \psi_2$, where $\psi_1, \psi_2 : G \rightarrow \left(\mathcal{O}/\left(\varpi \right) \right)^{\times}$ are characters. Then there exists a stable lattice $T \subset V$ for which $\bar{\rho}_{T}$ is the form $\begin{pmatrix} \psi_1 & * \\ 0 & \psi_2 \end{pmatrix}$ but is not semi-simple. 
\end{prp}

Let $f$ be a normalized Hecke eigen cusp form and $\rho_f : \mathrm{Gal}(\overline{\mathbf{Q}}/\mathbf{Q}) \rightarrow GL_2(K)$ the continuous irreducible representation attached to $f$, where $K$ is the field $\mathbf{Q}_p(\left\{a(n, f) \right\}_{n \geq 1})$. We denote by $\mathcal{L}(\rho_f)$ the set of the isomorphic classes of stable lattices of $\rho_f$. Since $\rho_f$ is irreducible, $\sharp \mathcal{L}(\rho_f)$ is finite (see (5) of Proposition \ref{2.c} below). The author wants to determine $\sharp \mathcal{L}(\rho_f)$ for a given $f$. For example the known result is obtained by Greenberg and Monsky for the Ramanujan's cusp form $\Delta=q\displaystyle \prod^{\infty}_{n=1}(1-q^n)^{24} \in S_{12}\left(\mathrm{SL}_2\left(\mathbf{Z} \right) \right)$ and $p=691$:

\begin{prp}[Greenberg, Monsky]\label{1.4}

Let $\rho_{\Delta} : \mathrm{Gal}(\overline{\mathbf{Q}}/\mathbf{Q}) \rightarrow GL_2(\mathbf{Q}_{691})$ be the $691$-adic representation attached to $\Delta$. Then $\sharp \mathcal{L}(\rho_{\Delta})=2$. 

\end{prp}

\begin{rem}
The work of Greenberg and Monsky is unpublished. See \cite[Section 12, Proposition 1]{Maz} for the statement. For the proof of more general settings,  see Proposition \ref{2.9} and the table after it. 
\end{rem}

A cusp form $f$ is called $p$-ordinary if its $p$-th Fourier coefficient $a\left(p, f \right)$ is a $p$-adic unit. Now we prepare some notations on Hida deformation. We fix a positive integer $N$ prime to $p$ and let $\chi$ be a primitive Dirichlet character modulo $Np$. If $\zeta \in \mu_{p^r} (r \geq 0)$ is a $p^r$-th root of unity, we denote by $\chi_{\zeta}$ the Dirichlet character as follows:
$$\chi_{\zeta} : \left(\mathbf{Z}/p^{r+1}\mathbf{Z} \right)^{\times} \rightarrow \overline{\mathbf{Q}}_p^{\times},\ u\ \mathrm{mod}\ p^{r+1} \mapsto \zeta.$$ Now let $\mathbb{I}$ be an integrally closed local domain which is finite flat over $\Lambda_{\chi}=\mathbf{Z}_p[\chi][[X]]$ and $\mathfrak{X}_{\mathbb{I}}$ the set of homomorphisms defined as follows:

$$\mathfrak{X}_{\mathbb{I}}=\Set{\varphi : \mathbb{I} \rightarrow \overline{\mathbf{Q}}_p | \varphi(1+X)=\zeta_{\varphi}u^{k_{\varphi}-2}, (k_{\varphi}, \zeta_{\varphi}) \in \mathbf{Z}_{\geq 2} \times \mu_{p^{\infty}} }.$$

Let $\mathscr{F}=\displaystyle \sum^{\infty}_{n=1}c(n, \mathscr{F})q^n \in \mathbb{I}[[q]]$ be an $\mathbb{I}$-adic cusp form (resp. $\mathbb{I}$-adic normalized Hecke eigen cusp form) with character $\chi$. That is,

$$f_{\varphi}:=\displaystyle \sum^{\infty}_{n=1}\varphi(c(n, \mathscr{F}))q^n \in S_{k_{\varphi}}\left(\Gamma_0(Np^{r_{\varphi}+1}), \chi_{\zeta_{\varphi}}\chi\omega^{1-k_{\varphi}}, \varphi\left(\mathbb{I} \right) \right)$$ is a $p$-ordinary cusp form (resp. $p$-ordinary normalized Hecke eigen cusp form) for all $\varphi \in \mathfrak{X}_{\mathbb{I}}$, where $\zeta_{\varphi}$ is a primitive $p^{r_{\varphi}}$-th root of unity. We denote by $S^{\mathrm{ord}}(\chi, \mathbb{I})$ the space of $\mathbb{I}$-adic forms with character $\chi$. Let $\mathbf{T}(\chi, \Lambda_{\chi})$ the ring generated over $\Lambda_{\chi}$ by all Hecke operators $T(l)$ for all primes $l$. Then $\mathbf{T}(\chi, \mathbb{I})=\mathbf{T}(\chi, \Lambda_{\chi}) \otimes_{\Lambda_{\chi}} \mathbb{I}$ acts on the space $S^{\mathrm{ord}}(\chi, \mathbb{I})$.

Let $\mathscr{F}$ be an $\mathbb{I}$-adic normalized Hecke eigen cusp form and $\mathrm{Frac}(\mathbb{I})$ the field of fraction of $\mathbb{I}$. Hida \cite{H2} proved that there is a continuous representation $$\rho_{\mathscr{F}} : \mathrm{Gal}(\overline{\mathbf{Q}}/\mathbf{Q}) \rightarrow \mathrm{GL}_2(\mathrm{Frac}(\mathbb{I}))$$ such that for any $\varphi \in \mathfrak{X}_{\mathbb{I}}$, the residual representation $\rho_{\mathscr{F}}\left(\mathrm{Ker}\varphi \right)$ (see Definition \ref{residual rep}) is isomorphic to $\rho_{f_{\varphi}}$.

From now on throughout the paper, we denote by $\phi$ the Euler function and we fix a positive integer $N$ prime to $p$. Let $\chi$ be a primitive Dirichlet character modulo $Np$. Let $\mathbb{I}$ be the same as above with $\mathfrak{m}$ the maximal ideal of $\mathbb{I}$. We denote by $S^{\mathrm{ord}}(\chi, \mathbb{I}), \mathbf{T}(\chi, \mathbb{I})$ the same as above. Let $\mathscr{F}$ be an $\mathbb{I}$-adic normalized Hecke eigen cusp form. Now we are going to determine $\sharp \mathcal{L}(\rho_{f_{\varphi}})$ when $\varphi$ varies in $\mathfrak{X}_{\mathbb{I}}$. Our result is the following theorem:

\begin{thm}\label{yandong}
Suppose $p \nmid \phi(N)$ and $\rho_{\mathscr{F}}\left(\mathfrak{m} \right) \cong \psi_1\oplus\psi_2$ such that $\psi_1$ (resp. $\psi_2$) is unramified (resp. ramified) at $p$. Assume the following condition:

\begin{list}{}{}
\item[(D)] There exist Dirichlet characters $\chi_1, \chi_2$ with relative prime conductors such that $\chi_1\chi_2=\chi, \chi_1\neq\chi_2\omega$ and $\overline{\chi}_i=\psi_i (i=1, 2).$
\end{list}
We enlarge $\mathbb{I}$ such that $\mathbb{I}$ is also finite flat over $\Lambda_{\chi_1\chi_2^{-1}}$. Then we have the following statements:

\begin{list}{}{}
\item[(1)]For any $\varphi \in \mathfrak{X}_{\mathbb{I}}$ such that $\varphi(1+X)=\zeta_{\varphi}u^{k_{\varphi}-2}$, we have $$\sharp \mathcal{L}(\rho_{f_{\varphi}}) \leq \mathrm{ord}_{\varpi_{\varphi}}(L_p(1-k_{\varphi}, \chi_{\zeta_{\varphi}}\chi_1^{-1}\chi_2\omega))+1,$$ where $\varpi_{\varphi}$ is a fixed uniformizer of $\varphi(\mathbb{I})$ and $L_p(s, \chi_{\zeta_{\varphi}}\chi_1^{-1}\chi_2\omega)$ is the Kubota-Leopoldt $p$-adic $L$-function. 

\item[(2)] Assume that 

\begin{list}{}{}
\item[(R)]$N=1$ and $\mathbf{T}(\chi, \Lambda_{\chi})$ is isomorphic to $\Lambda_{\chi}$. 
\end{list}

Then for any $\varphi \in \mathfrak{X}_{\mathbb{I}}$ such that $\varphi(1+X)=\zeta_{\varphi}u^{k_{\varphi}-2}$, we have $$\sharp \mathcal{L}(\rho_{f_{\varphi}})=\mathrm{ord}_{\varpi_{\varphi}}(L_p(1-k_{\varphi}, \chi_{\zeta_{\varphi}}\chi\omega))+1.$$

\item[(3)]Let $L_{\infty}, L_{\infty}\left(Np \right)$ be the maximal unramified abelian $p$-extension of $\mathbf{Q}\left(\mu_{Np^{\infty}} \right)$ and the maximal abelian $p$-extension unramified outside $Np$ of $\mathbf{Q}\left(\mu_{Np^{\infty}} \right)$. We denote by $X_{\infty}=\mathrm{Gal}\left(L_{\infty}/\mathbf{Q}\left(\mu_{Np^{\infty}} \right) \right)$ and by $Y_{\infty}=\mathrm{Gal}\left(L_{\infty}\left(Np \right)/\mathbf{Q}\left(\mu_{Np^{\infty}} \right) \right)$ on which $\Delta_{Np}=\mathrm{Gal}\left(\mathbf{Q}\left(\mu_{Np^{\infty}} \right)/\mathbf{Q}_{\infty} \right)$ acts by conjugation. Assume the following conditions:
\begin{list}{}{}

\item[(C)] The $\Lambda_{\chi_1\chi_2^{-1}}$-modules $X_{\infty}^{\chi_1\chi_2^{-1}}$ and $Y_{\infty}^{\chi_1^{-1}\chi_2}$ are cyclic. 
\item[(P)] The ideal generated by the Iwasawa power series (see Section 2.3 below) $\hat{G}_{\chi_1^{-1}\chi_2}(X)\mathbb{I}$ is a prime ideal in $\mathbb{I}$. 
\end{list}

Then for any $\varphi \in \mathfrak{X}_{\mathbb{I}}$ such that $\varphi(1+X)=\zeta_{\varphi}u^{k_{\varphi}-2}$, we have

$$\sharp \mathcal{L}(\rho_{f_{\varphi}})=\mathrm{ord}_{\varpi_{\varphi}}(L_p(1-k_{\varphi}, \chi_{\zeta_{\varphi}}\chi_1^{-1}\chi_2\omega))+1.$$

\end{list}

\end{thm}

Theorem \ref{yandong} will be proved at the end of Section 3.2. Now we discuss the boundedness of $\sharp\mathcal{L}\left(\rho_{f_{\varphi}} \right)$ when the weight and the level vary. When we fix an $r \in \mathbf{Z}_{\geq 0}$, we define $\mathfrak{X}^{(r)}_{\mathbb{I}}:$

$$\mathfrak{X}^{(r)}_{\mathbb{I}}=\Set{\varphi \in \mathfrak{X}_{\mathbb{I}} | \varphi(1+X)=\zeta_{\varphi}u^{k_{\varphi}-2}, (k_{\varphi}, \zeta_{\varphi}) \in \mathbf{Z}_{\geq 2} \times \left(\mu_{p^{\infty}} \setminus \mu_{p^r} \right) }.$$

When we fix a $\zeta \in \mu_{p^{\infty}}$, we define $\mathfrak{X}_{\mathbb{I}, \zeta}:$ $$\mathfrak{X}_{\mathbb{I}, \zeta}=\Set{ \varphi \in \mathfrak{X}_{\mathbb{I}} | \varphi(1+X)=\zeta u^{k_{\varphi}-2}, k_{\varphi} \geq 2 }.$$

When we fix a $k \in \mathbf{Z}_{\geq 2}$, we define $\mathfrak{X}_{\mathbb{I}, k}:$

$$\mathfrak{X}_{\mathbb{I}, k}=\Set{ \varphi \in \mathfrak{X}_{\mathbb{I}} | \varphi(1+X)=\zeta_{\varphi}u^{k-2}, \zeta_{\varphi} \in \mu_{p^{\infty}} }.$$

\begin{cor}\label{yandong2}
Let the assumptions and the notations be as in Theorem \ref{yandong}. We denote by $\hat{G}_{\chi_1^{-1}\chi_2}^{*}(X)$ the distinguished polynomial associated to $\hat{G}_{\chi_1^{-1}\chi_2}(X)$. Then we have the following statements:

\begin{list}{}{}
\item[(1)]There exists an integer $r \in \mathbf{Z}_{\geq 0}$ such that $$\sharp\mathcal{L}(\rho_{f_{\varphi}}) \leq \mathrm{rank}_{\Lambda_{\chi}}\mathbb{I}\cdot\mathrm{deg}\hat{G}_{\chi_1^{-1}\chi_2}^{*}(X) +1$$ is bounded when $\varphi$ varies in $\mathfrak{X}^{(r)}_{\mathbb{I}}$, where $\mathrm{rank}_{\Lambda_{\chi}}\mathbb{I}$ is the rank of the $\Lambda_{\chi}$-module $\mathbb{I}$.
\item[(2)]For each integer $k \geq 2$, $\sharp \mathcal{L}(\rho_{f_{\varphi}})$ is bounded when $\varphi$ varies in $\mathfrak{X}_{\mathbb{I}, k}$.

\item[(3)] Suppose that $\mathbb{I}$ is isomorphic to $\mathcal{O}[[X]]$ with $\mathcal{O}$ the ring of integers of a finite extension of $\mathbf{Q}_p$. Then there exists an integer $r^{\prime} \in \mathbf{Z}_{\geq 0}$ such that $\sharp\mathcal{L}(\rho_{f_{\varphi}})$ is constant when $\varphi$ varies in $\mathfrak{X}^{(r^{\prime})}_{\mathbb{I}}$. 
\item[(4)] Assume the condition (R) or both of the conditions (C) and (P). For each $\zeta \in \mu_{p^{\infty}}$, $\sharp\mathcal{L}(\rho_{f_{\varphi}})$ is unbounded when $\varphi$ varies in $\mathfrak{X}_{\mathbb{I}, \zeta}$ if and only if $L_p(1-s, \chi_{\zeta}\chi_1^{-1}\chi_2\omega)$ has a zero in $\mathbf{Z}_p$.

\end{list}

\end{cor}

Corollary \ref{yandong2} will be proved in Section 3.3. Let $\mathcal{L}(\rho_{\mathscr{F}})$ be the set of the isomorphic classes of stable lattices of Hida deformation $\rho_{\mathscr{F}}$. Now we give a result of $\sharp\mathcal{L}(\rho_{\mathscr{F}})$ answering Question 4.5 1 of \cite{Ochiai08}.

\begin{cor}\label{yandong3}
Let the assumptions and the notations be as in Theorem \ref{yandong}. Assume the conditions (D), (C) and (P). Further assume the following condition
\begin{list}{}{}
\item[(F)] There exists a stable lattice $\mathcal{T}$ which is free over $\mathbb{I}$.

\end{list}
Suppose that there exists a $\zeta \in \mu_{p^{\infty}}$ such that $L_p(1-s, \chi_{\zeta}\chi_1^{-1}\chi_2\omega)$ has a zero in $\mathbf{Z}_p$. Then $\sharp\mathcal{L}(\rho_{\mathscr{F}})=\infty$. 
\end{cor}

Corollary \ref{yandong3} will be proved in Section 3.4.

\begin{rem}
Mazur-Wiles \cite[\S 9]{MW2}, Tilouine \cite[Theorem 4.4]{T}  and Mazur-Tilouine \cite[\S 2, Corollary 6]{MT}  give a list of cases where the condition (F) is known to be true. 
\end{rem}

\begin{out}
The outline of this paper is as follows. In Section 2, we recall known results concerning the Bruhat-Tits tree of $\mathrm{GL}_2$, Hida deformation and Kubota-Leopoldt $p$-adic $L$-function. These will be used frequently in Section 3. In Section 3, first we determine the number of isomorphic classes of stable lattices in a given $p$-adic representation by using the Bella\"iche-Chenevier reducibility ideal $I\left(\rho \right)$. Then we give the proof of the main results. In Section 4, we give two examples of Hida deformations associated to an $\mathbb{I}$-adic normalized Hecke eigen cusp form $\mathscr{F} \in S^{\mathrm{ord}}\left(\omega^{k_0-1}, \mathbb{I} \right)$ when $\left(p, k_0 \right)=\left(691, 12 \right)$ and $\left(547, 486 \right)$.
\end{out}

\begin{ack}
The author expresses his sincere gratitude to Professor Tadashi Ochiai for his constant encouragement and spending a lot of time to read the manuscript carefully and pointing out mistakes. Thanks are also due to Kenji Sakugawa for reading the manuscript and correcting several mistakes. 
\end{ack}

\section{The Bruhat-Tits Tree, Hida deformation and $p$-adic $L$-function}

\subsection{The lattices and the Bruhat-Tits Tree}
Let $A$ be a commutative integral domain with field of fractions $K=\mathrm{Frac}(A)$ and $V$ a $n$-dimensional $K$-vector space. We say an $A$-submodule $T$ of $V$ is a lattice of $V$ if there exist two free $A$-submodules $L_1, L_2$ of $V$ such that $L_1 \subset T \subset L_2$ and $\mathrm{rank}_A L_1=n$. If $A$ is Noetherian, we have that an A-submodule $T$ of $V$ is a lattice of $V$ if and only if $T$ is finitely generated and $T \otimes_A K=V$ (see \cite[V\hspace{-.1em}I\hspace{-.1em}I, 4.1, Corollary to Proposition 1]{Bour}).
Now we assume $A=\mathcal{O}$ which is the ring of integers of a finite extension field of $\mathbf{Q}_p$ with a fixed uniformizer $\varpi$ of $\mathcal{O}$. For a given $p$-adic representation $$\rho : G \rightarrow \mathrm{Aut}_K(V)$$ of a compact group $G$, we say that $T$ is a $G$-stable lattice of $V$ if $T$ is a lattice and $\rho(G)T=T$. This means that $T$ is also an $\mathcal{O}[G]$-module. Since $G$ is compact, there exists a $G$-stable lattice in $V$ (see \cite{Seab} pp.1-2). We denote by $\rho_T$ the representation $$\rho_T : G \rightarrow \mathrm{Aut}_{\mathcal{O}} (T) \cong \mathrm{GL}_n(\mathcal{O})$$and by $\bar{\rho}_T$ the representation $\rho_T\ \mathrm{mod}\ \varpi$ as follows: $$\rho_T\ \mathrm{mod}\ \varpi : G \stackrel{\rho_{T}}{\rightarrow} \mathrm{Aut}_{\mathcal{O}} (T) \stackrel{\mathrm{mod}\ \varpi}{\longrightarrow} \mathrm{Aut}_{\mathcal{O}/\left(\varpi \right)}\left(T/\varpi T \right) \cong \mathrm{GL}_n\left(\mathcal{O}/\left(\varpi \right) \right).$$For stable lattices $T$ and $T^{\prime}$, the representation $\bar{\rho}_T, \bar{\rho}_{T^{\prime}}$ can be non-isomorphic to each other. However the semi-simplification $\bar{\rho}^{\mathrm{ss}}_T$ of $\bar{\rho}_T$ is isomorphic to $\bar{\rho}^{\mathrm{ss}}_{T^{\prime}}$ by the Brauer-Nesbitt theorem. We denote by $\bar{\rho}^{\mathrm{ss}}$ the semi-simplification of $\bar{\rho}_T$.

Now following \cite{SeTe} and \cite{Bell3}, we introduce the graph structure of lattices which will be used to prove Proposition \ref{2.8}. From now on to the end of this section we assume $n=2$.

For a lattice $T$ of $V$, we denote by $[T]=\set{xT | x \in K^{\times}}$ the equivalence class up to homotheties. Let $\mathcal{X}$ be the set of all $[T]$ where $T$ is a lattice. We say that a point $x^{\prime}$ in $\mathcal{X}$ is a neighbor of a point $x \in \mathcal{X}$ if $x^{\prime} \neq x$ and there are lattices $T, T^{\prime}$ of $V$ such that $x=[T], x^{\prime}=[T^{\prime}]$ and $\varpi T \subset T^{\prime} \subset T$. In this way one defines a combinatorial graph structure on $\mathcal{X}$.

\begin{thm}[{\cite[Chapter II, Theorem 1]{SeTe}}]\label{2.1}
The graph $\mathcal{X}$ is a tree.
\end{thm}

Now we recall some basic notions on the tree $\mathcal{X}$. Let $x, x^{\prime} \in \mathcal{X}$. A path without backtracking from $x$ to $x^{\prime}$, which is denoted by $\mathrm{Path}_{x, x^{\prime}}$ is a sequence $x=x_0, x_1, . . ., x_n=x^{\prime}$ of points in $\mathcal{X}$ such that $x_i$ is a neighbor of $x_{i+1}$ and $x_i \neq x_j$ if $i \neq j$. We define the integer $n=d(x, x^{\prime}) \geq 0$ to be the distance between $x$ and $x^{\prime}$. Let $x=[T]$ and we fix a positive integer $n$, then there is a natural bijection between the set of the points $x^{\prime}$ in $\mathcal{X}$ such that $d\left(x, x^{\prime} \right)=n$ and the set of lattices $\varpi^{n}T \subset T^{\prime} \subset T$ such that $T/T^{\prime} \cong \mathcal{O}/\left(\varpi \right)^n$ as an $\mathcal{O}$-module.

In a tree, we define the segment $[x, x^{\prime}]$ as 

\begin{eqnarray*}
[x, x^{\prime}] =\left\{ \begin{array}{ll}
\left\{x \right\} & (x=x^{\prime}) \\
\mathrm{Path}_{x, x^{\prime}} & (x \neq x^{\prime}) \\
\end{array} \right.
\end{eqnarray*}
A subset $C$ of $\mathcal{X}$ is called a convex if for every $x, x^{\prime} \in C$, the segment $[x, x^{\prime}] \subset C$. 

We denote by $\mathcal{X}(\rho)$ the set of $\mathcal{X}$ that are fixed by $\rho(G)$. We summarize some results on $\mathcal{X}\left(\rho \right)$:
\begin{prp}\label{2.c}\

\begin{list}{}{}

\item[(1)]$\mathcal{X}(\rho)$ is a convex (\cite[\S 3.1]{Bell3}).
\item[(2)]If $x \in \mathcal{X}(\rho)$, then $x$ has no neighbor in $\mathcal{X}(\rho)$ if and only if $\overline{\rho}_x$ is irreducible (\cite[Proposition 11 (d)-(i)]{Bell3}).
\item[(3)]If $x \in \mathcal{X}(\rho)$, then $x$ has exactly one neighbor in $\mathcal{X}(\rho)$ if and only if $\overline{\rho}_x$ is reducible but indecomposable (\cite[Proposition 11 (d)-(ii)]{Bell3}). 
\item[(4)]If $x \in \mathcal{X}(\rho)$, then $x$ has exactly two neighbors in $\mathcal{X}(\rho)$ if and only if $\overline{\rho}_x$ is decomposed into two distinct characters (\cite[Proposition 11 (d)-(iii)]{Bell3}).
\item[(5)]$\rho$ is irreducible if and only if $\mathcal{X}(\rho)$ is bounded (\cite[Lemme 10]{Bell3}).
\item[(6)]Assume that $\rho$ is irreducible and $\bar{\rho}^{\mathrm{ss}} \cong \psi_1 \oplus \psi_2$ of characters $\psi_1, \psi_2 : G \rightarrow {\left(\mathcal{O}/\left(\varpi \right) \right)}^{\times}$ with $\psi_1 \neq \psi_2$. Then $\mathcal{X}(\rho)$ is a segment.

\end{list}

\end{prp}

The assertion (6) easily follows from the assertions (1), (4) and (5) (cf. \cite{Bell1} the arguments before \S 1.3 in page 7).

\subsection{$\mathbb{I}$-adic forms and Galois representations}

In this section, we review some fundamental results on $\mathbb{I}$-adic cusp forms and their Galois representations. For more detail on this theory, the reader can refer to Chapter 7 of \cite{H3}.

Recall that $\mathbb{I}$ is an integrally closed local domain which is finite flat over $\Lambda_{\chi}=\mathbf{Z}_p[\chi][[X]]$, where $\chi$ is a primitive Dirichlet character modulo $Np$.   We denote by $\mathfrak{X}_{\Lambda_{\chi}}$ and $\mathfrak{X}_{\mathbb{I}}$ the sets of homomorphisms defined as follows:

$$\mathfrak{X}_{\Lambda_{\chi}}=\left\{\nu_{k, \zeta} : \Lambda_{\chi} \rightarrow \overline{\mathbf{Q}}_p\  \mid\ \nu_{k, \zeta}(1+X)=\zeta u^{k-2}, (k, \zeta) \in \mathbf{Z}_{\geq 2} \times \mu_{p^{\infty}} \right\},$$

$$\mathfrak{X}_{\mathbb{I}}=\left\{\varphi : \mathbb{I} \rightarrow \overline{\mathbf{Q}}_p\  \mid\ \varphi\mid_{\Lambda_{\chi}}=\nu_{k_{\varphi}, \zeta_{\varphi}}, \nu_{k_{\varphi}, \zeta_{\varphi}} \in \mathfrak{X}_{\Lambda_{\chi}}  \right\}.$$If $\zeta \in \mu_{p^r} (r \geq 0)$ is a primitive $p^r$-th root of unity, let $$\chi_{\zeta} : \left(\mathbf{Z}/p^{r+1}\mathbf{Z} \right)^{\times} \rightarrow \overline{\mathbf{Q}}_p^{\times}, u\ \mathrm{mod}\ p^{r+1} \mapsto \zeta$$ be the character associated to $\zeta$.

\begin{dfn}\label{3.1}

We call $\mathscr{F}=\displaystyle \sum^{\infty}_{n=1}c(n, \mathscr{F})q^n \in \mathbb{I}[[q]]$ an $\mathbb{I}$-adic form (resp. $\mathbb{I}$-adic normalized Hecke eigen cusp form) with Dirichlet character $\chi$ if for each $\varphi \in \mathfrak{X}_{\mathbb{I}}$ with $\varphi\mid_{\Lambda_{\chi}}=\nu_{k_{\varphi}, \zeta_{\varphi}}$ such that $\zeta_{\varphi}$ is a primitive $p^{r_{\varphi}}$-th root of unity, $$f_{\varphi}:=\displaystyle \sum^{\infty}_{n=1}\varphi(c(n, \mathscr{F}))q^n \in S_{k_{\varphi}}(\Gamma_0(Np^{r_{\varphi}+1}), \chi_{\zeta_{\varphi}}\chi\omega^{1-k_{\varphi}}, \varphi(\mathbb{I}))$$ is the $q$-expansion of a $p$-ordinary cusp form (resp. $p$-ordinary normalized Hecke eigen cusp form).
\end{dfn}

We also denote by $\varphi\left(\mathscr{F} \right)$ the above cusp form $f_{\varphi}$. Recall that $S^{\mathrm{ord}}\left(\chi, \mathbb{I} \right)$ is the space of $\mathbb{I}$-adic forms with Dirichlet character $\chi$.

\begin{thm}[Hida {\cite[\S 7.4, Theorem 7]{H3}}]\label{Hida}
Let $\zeta \in \mu_{p^r} (r \geq 0)$ be a primitive $p^r$-th root of unity and $\chi_{\zeta}$ the character associated to $\zeta$. Let $$f \in S_{k}\left(\Gamma_0\left(Np^{r+1} \right), \chi_{\zeta}\chi\omega^{1-k}, \overline{\mathbf{Q}} \right)$$ be a $p$-ordinary normalized Hecke eigen cusp form of weight $k \geq 1$. Then there exist an integrally closed local domain $\mathbb{I}$ which is finite flat over $\Lambda_{\chi}$, an $\mathbb{I}$-adic normalized Hecke eigen cusp form $\mathscr{F} \in S^{\mathrm{ord}}\left(\chi, \mathbb{I} \right)$ and a $\varphi \in \mathfrak{X}_{\mathbb{I}}$ such that $\varphi(\mathscr{F})=f$.

\end{thm}

\begin{dfn}\label{lat}
A Galois representation $\rho : \mathrm{Gal}(\overline{\mathbf{Q}}/\mathbf{Q}) \longrightarrow \mathrm{GL}_2(\mathrm{Frac}(\mathbb{I}))$ is continuous if there exists a lattice $\mathcal{T} \subset \mathrm{Frac}(\mathbb{I})^{\oplus 2}$ which is stable under $\mathrm{Gal}(\overline{\mathbf{Q}}/\mathbf{Q})$-action such that $\rho : \mathrm{Gal}(\overline{\mathbf{Q}}/\mathbf{Q}) \longrightarrow \mathrm{Aut}_{\mathbb{I}}(\mathcal{T})$ is continuous with respect to the topology of $\mathcal{T}$ defined by $\mathfrak{m}$ the maximal ideal of $\mathbb{I}$.
\end{dfn}

Hida associates a continuous Galois representation over $\mathrm{Frac}(\mathbb{I})$ to an $\mathbb{I}$-adic normalized Hecke eigen cusp form $\mathscr{F}$ as follows:

\begin{thm}[Hida {\cite[Theorem 2.1]{H2}}]\label{3.3}
There exists a continuous irreducible representation $\rho_{\mathscr{F}} : \mathrm{Gal}(\overline{\mathbf{Q}}/\mathbf{Q}) \longrightarrow \mathrm{GL}_2(\mathrm{Frac}(\mathbb{I}))$ with the following properties:
\begin{list}{}{}
\item[1.] $\rho_{\mathscr{F}}$ is unramified outside $Np$.
\item[2.] For the geometric Frobenius element $\mathrm{Frob}_l$ at $l \nmid Np$, we have:
$$\mathrm{tr}\rho_{\mathscr{F}}(\mathrm{Frob}_l)=c(l, \mathscr{F}),$$
$$\mathrm{det}\rho_{\mathscr{F}}(\mathrm{Frob}_l)=\chi(l)\left(u\left(1+X \right) \right)^{s_l},$$
where $d=\omega(d)(1+p)^{s_d}$ under the isomorphism $\mathbf{Z}_p^{\times} \cong \mu_{p-1} \times (1+p\mathbf{Z}_p).$
\end{list}
\end{thm}

We have the following local property due to Mazur and Wiles: 

\begin{thm}[Wiles {\cite[Theorem 2.2.2]{Wi88}}]\label{w}
With the same notations as above, the restriction of $\rho_{\mathscr{F}}$ to the decomposition group $D_p=\mathrm{Gal}\left(\overline{\mathbf{Q}}_p/\mathbf{Q}_p \right)$ is given up to equivalence by 

$$\rho_{\mathscr{F}}\mid_{D_p} \sim \begin{pmatrix} \varepsilon_1 & 0 \\ * & \varepsilon_2 \end{pmatrix}$$with $\varepsilon_1$ unramified and $\varepsilon_1(\mathrm{Frob}_p)=c(p, \mathscr{F})$. 

\end{thm}

\begin{dfn}\label{residual rep}
For a prime ideal $\mathcal{P}$ of $\mathbb{I}$, a Galois representation $$\rho_{\mathscr{F}}\left(\mathcal{P} \right) : \mathrm{Gal}\left(\overline{\mathbf{Q}}/\mathbf{Q} \right) \rightarrow \mathrm{GL}_2\left(\mathrm{Frac}\left(\mathbb{I}/\mathcal{P} \right) \right)$$ is called a residual representation of $\rho_{\mathscr{F}}$ modulo $\mathcal{P}$ if $\rho_{\mathscr{F}}\left(\mathcal{P} \right)$ is semi-simple, continuous under the $\mathfrak{m}$-adic topology of $\mathrm{Frac}\left(\mathbb{I}/\mathcal{P} \right)$ and satisfies the following properties: 
\begin{list}{}{}
\item[1.]$\rho_{\mathscr{F}}\left(\mathcal{P} \right)$ is unramified outside $Np$.
\item[2.]For the geometric Frobenius element $\mathrm{Frob}_l$ at $l \nmid Np$,
$$\mathrm{tr}\rho_{\mathscr{F}}(\mathcal{P})(\mathrm{Frob}_l)=c(l, \mathscr{F})\ \mathrm{mod}\ \mathcal{P},$$
$$\mathrm{det}\rho_{\mathscr{F}}(\mathcal{P})(\mathrm{Frob}_l)=\chi(l)\left(u\left(1+X \right) \right)^{s_l}\ \mathrm{mod}\ \mathcal{P}.$$
\end{list}
\end{dfn}

Although $\rho_{\mathscr{F}}$ may not have $\mathrm{Gal}(\overline{\mathbf{Q}}/\mathbf{Q})$-stable lattice which is isomorphic to ${\mathbb{I}}^{\oplus 2}$, the following fact is well-known (see {\cite[\S 7.5, Corollary 1]{H3}} for example).

\begin{prp}\label{3.5}
For every prime ideal $\mathcal{P}$, the residual representation $\rho_{\mathscr{F}}(\mathcal{P})$ exists and is unique up to isomorphism over an algebraic closure of $\mathrm{Frac}(\mathbb{I}/\mathcal{P})$. 

\end{prp}

\subsection{Kubota-Leopoldt $p$-adic $L$-function}

Now we recall some facts about the Kubota-Leopoldt $p$-adic $L$-function. Let $\psi$ be an arbitrary Dirichlet character. Kubota-Leopoldt (see \cite[\S 3, Theorem 2]{Iwa}) showed that there exists a $p$-adic continuous function $L_p(s, \psi)$ for $s \in \mathbf{Z}_p-\left\{1 \right\}$ (also continuous at $s=1$ if $\psi$ is non-trivial) with the following interpolate property for $k \geq 1$:

$$L_p(1-k, \psi)=(1-\psi\omega^{-k}(p)p^{k-1})L(1-k, \psi\omega^{-k}).$$

Set

\begin{eqnarray*}
H_{\psi}(X) =\left\{ \begin{array}{ll}
\psi(u)(1+X)-1 & \mathrm{if}\ \psi\ \mathrm{factors\ through}\ \mathrm{Gal}\left(\mathbf{Q}_{\infty}/\mathbf{Q} \right). \\
1 & \mathrm{otherwise}. \\
\end{array} \right.
\end{eqnarray*}

Iwasawa \cite[\S 6]{Iwa} showed that there exists a unique power series $G_{\psi}(X) \in \mathbf{Z}_p[\psi][[X]]$ such that 

\begin{list}{}{}
\item[(i)]$L_p(1-s, \psi)=G_{\psi}(u^s-1)/{H_{\psi}(u^s-1)},$
\item[(ii)]if $\rho$ factors through $\mathrm{Gal}\left(\mathbf{Q}_{\infty}/\mathbf{Q} \right)$, then $G_{\psi\rho}(X)=G_{\psi}(\rho(u)(1+X)-1).$

\end{list}

We define
\begin{list}{}{}
\item[]$\hat{G}_{\psi}(X)=G_{\psi\omega}(u^2(1+X)-1),$
\item[]$\hat{H}_{\psi}(X)=H_{\psi\omega}(u^2(1+X)-1)$
\end{list}

for later reference.

\section{Proof of Theorem \ref{yandong} and its corollaries}

\subsection{Calculation of $\sharp \mathcal{L}(\rho)$ by means of $I(\rho)$}
For a given $p$-adic representation $\rho$, recall that $\mathcal{L}(\rho)$ is the set of the isomorphic classes of stable lattices of $\rho$. First we determine $\sharp \mathcal{L}(\rho)$ in this section. The following lemma is proved by Bella\"iche-Chenevier \cite{Bell2}.

\begin{lem}[Bella\"iche-Chenevier {\cite[Lemme 1]{Bell2}}]\label{2.6}
	
Let $(A, \mathfrak{m}, k)$ be a complete local domain such that $\mathrm{char}(k) \neq 2$, where $\mathfrak{m}$ is the maximal ideal of $A$ and $k$ the residue field $A/\mathfrak{m}$. Let $\rho : G \rightarrow \mathrm{GL}_2(\mathrm{Frac}(A))$ be a linear representation of a group $G$ satisfying $\mathrm{tr}\rho(G) \subset A$ and $\mathrm{tr}\rho\ \mathrm{mod}\ \mathfrak{m}=\psi_1 + \psi_2, \psi_1 \neq \psi_2$, where $\psi_1, \psi_2 : G \rightarrow k^{\times}$ are characters. Let $g_0 \in G$ be an element satisfying $\psi_1(g_0) \neq \psi_2(g_0)$ and $\lambda_1, \lambda_2 \in A$  the roots of the characteristic polynomial of $\rho(g_0)$. Choose a basis $\left\{e_1, e_2 \right\}$ of the representation $\rho$ such that $\rho(g_0)e_i=\lambda_i e_i\ (i=1,2)$. Write $\rho(g) = \begin{pmatrix} a(g) & b(g) \\ c(g) & d(g) \end{pmatrix}$ for any $g \in G$. 
	
Let $I \subsetneq A$ be an ideal such that there exist two characters $\vartheta_1, \vartheta_2 : G \rightarrow (A/I)^{\times}$ such that $$\mathrm{tr}\rho(g)\ \mathrm{mod}\ I = \vartheta_1(g) + \vartheta_2(g)$$ for any $g \in G$. Assume $\vartheta_1\ \mathrm{mod}\ \mathfrak{m}=\psi_1, \vartheta_2\ \mathrm{mod}\ \mathfrak{m}=\psi_2$ without loss of generality. Then for any $g, g^{\prime} \in G$, we have $a(g), d(g) \in A$, $a(g)\ \mathrm{mod}\ I=\vartheta_1(g), d(g)\ \mathrm{mod}\ I=\vartheta_2(g)$, and $b(g)c(g^{\prime}) \in I$. 
\end{lem}

\begin{rem}\label{3.2} If $\mathrm{char}(k)=2$, the statement holds assuming an extra condition on the determinate (cf. \cite[Lemme 1]{Bell2}).

\end{rem}

\begin{dfn}\label{3.2.2}
Let $(A, \mathfrak{m}, k)$ be a complete local domain such that $\mathrm{char}(k) \neq 2$, where $\mathfrak{m}$ is the maximal ideal of $A$ and $k$ the residue field. Let $\rho : G \rightarrow \mathrm{GL}_2(\mathrm{Frac}(A))$ be a linear representation of a  group $G$ satisfying $\mathrm{tr}\rho(G) \subset A$ and $\mathrm{tr}\rho\ \mathrm{mod}\ \mathfrak{m}=\psi_1 + \psi_2, \psi_1 \neq \psi_2$, where $\psi_1, \psi_2 : G \rightarrow k^{\times}$ are characters.  For any $g \in G$, write $\rho(g) = \begin{pmatrix} a(g) & b(g) \\ c(g) & d(g) \end{pmatrix}$ with respect to the basis taken as in Lemma \ref{2.6}. We define $I(\rho)$ the ideal of $A$ which is generated by $b(g)c(g^{\prime})$ for all $g, g^{\prime} \in G$.

\end{dfn}

The ideal $I(\rho)$ is well-defined by Lemma \ref{2.6}. Under the above preparation, we are ready to determine $\sharp\mathcal{L}(\rho)$ of a $p$-adic representation $\rho$.

\begin{prp}\label{2.8}
Let $\left(\mathcal{O}, \varpi, \mathcal{O}/\left(\varpi \right) \right)$ be the ring of integers of a finite extension of $\mathbf{Q}_p$, where $\varpi$ is a fixed uniformizer of $\mathcal{O}$. Let $V$ be a vector space of dimension 2 over $K=\mathrm{Frac}(\mathcal{O})$ and $\rho : G \rightarrow \mathrm{Aut}_K(V)$ a continuous irreducible representation of a compact group $G$. 

Assume that $$\mathrm{tr}\rho\ \mathrm{mod}\ \mathfrak{\varpi}=\psi_1 + \psi_2, \psi_1 \neq \psi_2,$$
where $\psi_1, \psi_2 : G \rightarrow (\mathcal{O}/\varpi)^{\times}$ are characters. Then we have 
$$\mathrm{ord}_{\varpi} I(\rho) +1=\sharp \mathcal{X}(\rho)=\sharp \mathcal{L}(\rho).$$
\end{prp}

This proposition is a special case of Bella\"iche-Graftieaux \cite[Th\'eor\`eme 4.1.3]{BG} (see also the remark immediately after it), but we give the proof here for self-containing.

\begin{proof}
	
	We first show $\mathrm{ord}_{\varpi} I(\rho) +1=\sharp \mathcal{X}(\rho)$. Fix a $g_0 \in G$ such that $\psi_1(g_0) \neq \psi_2(g_0)$. The characteristic polynomial of $\rho(g_0)$ $$X^2-\mathrm{tr}\rho(g_0)X+\mathrm{det}\rho(g_0)$$ has roots $\lambda_1 \neq \lambda_2$ in $A$ such that $\lambda_i\ \mathrm{mod}\ \varpi=\psi_i(g_0) (i=1, 2)$ by Hensel's lemma. Choose a basis $\left\{e_1, e_2 \right\}$ of the representation $\rho$ such that $\rho(g_0)e_i=\lambda_i e_i\ (i=1,2)$. Write $\rho(g) = \begin{pmatrix} a(g) & b(g) \\ c(g) & d(g) \end{pmatrix}$ for all $g \in G$. Let $B$ be the module of $\mathcal{O}$ generated by $b(g)$ for all $g \in G$. Since $\rho$ is irreducible, we have $B \neq \left(0 \right)$. Since $\rho$ is continuous and $G$ is compact, there exists a stable lattice. This implies $B=(\varpi)^{r}$ for an integer $r$. 

If we replace $\rho$ by $\begin{pmatrix} 1 & 0 \\ 0 & \varpi^r \end{pmatrix}\rho\begin{pmatrix} 1 & 0 \\ 0 & \varpi^r \end{pmatrix}^{-1}$ and we denote by the same symbol $\rho$ for this new representation. Then we have the following properties for the new $\rho$:

 \begin{list}{}{}
		
		\item[(1)]$\rho$ takes values in $\mathrm{GL}_2(\mathcal{O})$ by Lemma \ref{2.6}.
		\item[(2)]$\rho(g_0)=\begin{pmatrix} \lambda_1 & 0 \\ 0 & \lambda_2 \end{pmatrix}$.
		\item[(3)]For any $g \in G$, write $\rho(g) = \begin{pmatrix} a(g) & b(g) \\ c(g) & d(g) \end{pmatrix}.$ There exists a $h \in G$ such that $\varpi \nmid b(h)$.
		
	\end{list}
	
We denote by the same symbol $B$ (resp. $C$) the ideal of $\mathcal{O}$ which is generated by new $b(g)$ (resp. new $c(g)$) for all $g \in G$. Since $BC=I(\rho)$ by Lemma \ref{2.6} and $B=\mathcal{O}$ by (1), we must have $C=I(\rho)=(\varpi)^n$ for a positive integer $n$. This also means that we have chosen a stable lattice $T$ such that $$\rho=\rho_{T} : G \rightarrow \mathrm{GL}_2(\mathcal{O})$$ and $T/\varpi T$ is not semi-simple. By reduction mod $\left(\varpi \right)^i\ (i=1, 2, \ldots , n)$, we obtain the $G$-stable lattices $T_1, \cdots, T_n$ such that $[T_i] \neq [T_j]$ if $i \neq j$. Then $n+1=\mathrm{ord}_{\varpi} I(\rho) +1 \leq \sharp \mathcal{X}(\rho)$. 
	
Let $\sharp \mathcal{X}(\rho)=m+1$. We have $\mathcal{X}(\rho)$ is a segment $[x, x_m]$ by (6) of Proposition \ref{2.c}. Let $T, T_m$ be the representatives of $x, x_m$ such that $T_m \subset T$ and $T/T_m \cong \mathcal{O}/\left(\varpi \right)^m$ as an $\mathcal{O}$-module. Hence there exists a basis of $T$ such that $\rho_{T} : G \rightarrow \mathrm{GL}_2(\mathcal{O}), g \mapsto \begin{pmatrix} a(g) & b(g) \\ c(g) & d(g) \end{pmatrix}$ satisfies $\varpi^m\mid c(g)$ for any $g \in G$. Then $$a\ \mathrm{mod}\ \varpi^m : G \rightarrow \left(\mathcal{O}/\left(\varpi \right)^m \right)^{\times}, g \mapsto a(g)\ \mathrm{mod}\ \varpi^m$$ and $$d\ \mathrm{mod}\ \varpi^m : G \rightarrow \left(\mathcal{O}/\left(\varpi \right)^m\right)^{\times}, g \mapsto d(g)\ \mathrm{mod}\ \varpi^m$$are two characters. Thus $I\left(\rho \right) \subset \left(\varpi \right)^m$ by Lemma \ref{2.6} and $\sharp \mathcal{X}(\rho) \leq \mathrm{ord}_{\varpi} I(\rho) +1$.

	Next we prove $\sharp \mathcal{X}(\rho)=\sharp \mathcal{L}(\rho)$. Suppose $\sharp \mathcal{X}(\rho)=n+1$. Since $\mathcal{X}\left(\rho \right)$ is a segment, by Section 2.1 there exist $T_0, T_1, \ldots , T_n$ the representatives of the points in $\mathcal{X}\left(\rho \right)$ such that
   \begin{list}{}{}
		\item[(i)] $[T_i]$ is a neighbor of $[T_{i-1}]$ and $T_0/T_i \cong \mathcal{O}/\left(\varpi \right)^i$ as an $\mathcal{O}$-module for $i=1, \cdots, n$.
		\item[(ii)] $T_0, T_n$ are mod $\varpi$ not semi-simple lattices and the others are not.

	\end{list}
Thus it is sufficient to show that for $i \neq j$, $T_i$ and $T_j$ are non-isomorphic to each other as $\mathcal{O}[G]$-modules. 
	
\begin{list}{}{}
\item[1.] Suppose we have $f : T_0 \stackrel{\sim}{\rightarrow} T_n$ as $\mathcal{O}[G]$-modules. Then $\varpi T_n \subset f\left(T_1 \right) \subset T_n$ since $[T_1]$ is a neighbor of $[T_0]$. Since $T_n$ is a mod $\varpi$ not semi-simple lattice, we have $f\left(T_1 \right)=\varpi T_{n-1}$ by (3) of Proposition \ref{2.c}. Hence $$T_1/\varpi T_0 \cong \varpi T_{n-1}/\varpi T_n \cong \mathcal{O}/\left(\varpi \right)[\psi_1]\ (\mathrm{resp}.\ \mathcal{O}/\left(\varpi \right)[\psi_2])$$

as an $\mathcal{O}[G]$-module. Thus there is no mod $\varpi$ not semi-simple stable lattice $T$ such that $T/\varpi T$ has a submodule which is isomorphic to $\mathcal{O}/\left(\varpi \right)[\psi_2]$ (resp. $\mathcal{O}/\left(\varpi \right)[\psi_1]$). This contradicts to the Ribet's lemma (Proposition \ref{1.2}).

\item[2.]Suppose we have $f : T_i \stackrel{\sim}{\rightarrow} T_j$ as $\mathcal{O}[G]$-modules for some $0<i < j<n$.  Since $T_0$ is a mod $\varpi$ not semi-simple lattice and $T_0, T_n$ are non-isomorphic as $\mathcal{O}[G]$-modules, we have $f\left(\varpi^i T_0 \right)=\varpi^l T_0$. Hence $$\mathcal{O}/\left(\varpi \right)^i \cong T_0/T_i \cong T_0/\varpi^{i-l}T_j$$as an $\mathcal{O}$-module. This implies $d\left([T_0], [T_j] \right)=i$. On the other hand, $[T_0]$ is an edge of the segment $\mathcal{X}\left(\rho \right)$, there exists an unique point $y \in \mathcal{X}\left(\rho \right)$ such that $d\left([T_0], y \right)=i$. This contradicts to $i \neq j$. 
		
\end{list}
\end{proof}

Now we give an example by using  Proposition \ref{2.8} to determine $\sharp \mathcal{L}(\rho_f)$, where $\rho_f$ is the Galois representation attached to a normalized Hecke eigen cusp form $f$. Let $\Delta \in S_{12}(\mathrm{SL}_2(\mathbf{Z}))$ be the Ramanujan's cusp form, whose $q$-expansion is equal to $q\displaystyle \prod^{\infty}_{n=1}(1-q^n)^{24}=\displaystyle \sum^{\infty}_{n=1}\tau(n)q^n $ and $$\rho_{\Delta} : \mathrm{Gal}\left(\overline{\mathbf{Q}}/\mathbf{Q} \right) \rightarrow \mathrm{GL}_2(\mathbf{Q}_p)$$ the Galois representation attached to $\Delta$.

\begin{prp}\label{2.9}
	
	The ideal $I(\rho_{\Delta}) \subset \mathbf{Z}_p$ defined as in Definition \ref{3.2.2} is the minimal ideal such that there exists an integer $a \in \mathbf{Z}$ such that for any prime $l \neq p$, 
	$$\tau(l) \equiv l^a + l^{11-a}\ \mathrm{mod}\ I(\rho_{\Delta}).$$
	
\end{prp}

\begin{proof}
	
	We denote by $\mathbf{Q}_{\left\{p, \infty \right\}}$ be the maximal Galois extension of $\mathbf{Q}$ which is unramified outside $\set{p, \infty}$ and by $G_{\left\{p, \infty \right\}}=\mathrm{Gal}\left(\mathbf{Q}_{\left\{p, \infty \right\}}/\mathbf{Q} \right)$. Since $\rho_{\Delta}$ is unramified outside $\set{p, \infty}$, $\rho_{\Delta}$ must factor through $G_{\left\{p, \infty \right\}}$.

Let $\vartheta_1, \vartheta_2 : G_{\left\{p, \infty \right\}} \rightarrow (\mathbf{Z}_p/I(\rho_{\Delta}))^{\times}$be the character such that $$\mathrm{tr}\rho\ \mathrm{mod}\ I(\rho_{\Delta})=\vartheta_1+\vartheta_2.$$ 
Since $\rho_{\Delta}$ is unramified outside $p$, $\vartheta_1$ and $\vartheta_2$ must factor through $\mathrm{Gal}\left(\mathbf{Q}\left(\mu_{p^{\infty}} \right)/\mathbf{Q} \right)$ by class field theory.  Thus $\vartheta_1$ and $\vartheta_2$ must be power of mod $I(\rho_{\Delta})$ $p$-adic cyclotomic character $\chi_{cyc}\ \mathrm{mod}\ I\left(\rho_{\Delta} \right)$. For the geometric Frobenius element $\mathrm{Frob}_l$ with prime $l \neq p$, we have $\chi_{cyc}(\mathrm{Frob}_l)=l$ and $\mathrm{det}(\mathrm{Frob}_l)=l^{11}$. Thus the proposition follows by the Chebotarev's density theorem.

\end{proof}

Serre and Swinnerton-Dyer showed that $\bar{\rho}_{\Delta}$ is reducible if and only if $p=2, 3, 5, 7$ and $691$ (see \cite[Corollary to Theorem 4]{SWin1}). \cite{SWin2} also showed the congruence mod $p^n$ for $p=3, 5, 7$ and $691$ (see \cite{SWin2}, page 77 for $p=691$, Theorem 4 for $p=5, 7$ and the table after Theorem 6 for $p=3$). Then combined with our arguments, we have the following table for odd primes. 

\begin{center}
	\begin{tabular}{|l|r|r|r|r|r|} \hline
		$p$ & 3 & 5 & 7 & 691 \\ \hline
		$\sharp \mathcal{L}(\rho_{\Delta})$ & 7 & 4 & 2 & 2 \\ \hline
	\end{tabular}
\end{center}

\subsection{The relation between $I(\rho_{\mathscr{F}})$ and the Iwasawa power series}

Let us take an $\mathbb{I}$-adic normalized Hecke eigen cusp form $\mathscr{F}$. Let the notations and the assumptions be as in Theorem \ref{yandong}. We denote by $\gamma$ a topological generator of $\mathrm{Gal}\left(\mathbf{Q}_{\infty}/\mathbf{Q} \right)$ such that $\kappa_{\mathrm{cyc}}(\gamma)=u$ 
and by $\kappa$ the universal cyclotomic character as follows: $$\kappa : \mathrm{Gal}\left(\overline{\mathbf{Q}}/\mathbf{Q} \right) \twoheadrightarrow \mathrm{Gal}\left(\mathbf{Q}_{\infty}/\mathbf{Q} \right)  \stackrel{\sim}{\rightarrow} 1+p\mathbf{Z}_p \hookrightarrow \Lambda_{\chi}^{\times},$$ where $1+p\mathbf{Z}_p \hookrightarrow \Lambda_{\chi}^{\times}$ is the homomorphism defined by sending $u$ to $1+X$. Write $\eta=\mathrm{det}\rho_{\mathscr{F}}$ for short. By Theorem \ref{w}, we have $$\rho_{\mathscr{F}}\mid_{D_p} \sim \begin{pmatrix} \varepsilon_1 & 0 \\ * & \varepsilon_2 \end{pmatrix},$$with $\varepsilon_1$ unramified. Recall that $\rho_{\mathscr{F}}\left(\mathfrak{m} \right)\cong\overline{\chi}_1\oplus\overline{\chi}_2$. Then for any $g \in D_p$, $\set{\overline{\varepsilon}_1(g), \overline{\varepsilon}_2(g)}$ and $\set{\overline{\chi}_1(g), \overline{\chi}_2(g)}$ are the set of the roots of the $\mathrm{mod}\ \mathfrak{m}$ characteristic polynomial of $\rho_{\mathscr{F}}(g)$: $X^2-\mathrm{tr}\rho_{\mathscr{F}}(g)X+\mathrm{det}\rho_{\mathscr{F}}(g)\ \mathrm{mod}\ \mathfrak{m}$, hence they must be coincide. Thus $\overline{\varepsilon}_1=\overline{\chi}_1 \mid_{D_p}$ and $\overline{\varepsilon}_2=\overline{\chi}_2 \mid_{D_p}$ under the assumption that $\overline{\chi}_1$ (resp. $\overline{\chi}_2$) is unramified (resp. ramified). We denote by $I_p$ the inertia group of $p$ and we choose a $g_0 \in I_p$ such that $\overline{\chi}_1(g_0) \neq \overline{\chi}_2(g_0)$.

Let $\left\{e_1, e_2 \right\}$ be a basis of $\mathrm{Frac}\left(\mathbb{I} \right)^{\oplus 2}$ such that 
\begin{equation}\label{0118}
\rho_{\mathscr{F}}(g_0) = \begin{pmatrix} 1 & 0 \\ 0 & \varepsilon_2(g_0) \end{pmatrix}, \rho_{\mathscr{F}}\mid_{D_p}= \begin{pmatrix} \varepsilon_1 & 0 \\ * & \varepsilon_2 \end{pmatrix}.
\end{equation}
Write $\rho_{\mathscr{F}}(g) = \begin{pmatrix} a(g) & b(g) \\ c(g) & d(g) \end{pmatrix}$ for any $g \in \mathrm{Gal}(\overline{\mathbf{Q}}/\mathbf{Q})$. We have $a(g), d(g)$ and  $b(g)c(g^{\prime}) \in \mathbb{I}$ for any $g, g^{\prime} \in \mathrm{Gal}(\overline{\mathbf{Q}}/\mathbf{Q})$ by Lemma \ref{2.6}. Recall that $I(\rho_{\mathscr{F}})$ is the ideal of $\mathbb{I}$ generated by $b(g)c(g^{\prime})$ for all $g, g^{\prime} \in \mathrm{Gal}(\overline{\mathbf{Q}}/\mathbf{Q})$. Since $\rho_{\mathscr{F}}\left(\mathfrak{m} \right)$ is reducible, we have $I\left(\rho_{\mathscr{F}} \right) \subset \mathfrak{m}$ by Lemma \ref{2.6}.

\begin{lem}\label{1}
Let us take the basis of $\mathrm{Frac}(\mathbb{I})^{\oplus 2}$ to be the same as the beginning of this section. For any $\varphi \in \mathfrak{X}_{\mathbb{I}}$, let $\varpi_{\varphi}$ be a fixed uniformizer of $\varphi(\mathbb{I})$. Then

$$\mathrm{ord}_{\varpi_{\varphi}}(\varphi (I(\rho_{\mathscr{F}})))+1=\sharp \mathcal{L}(\rho_{f_{\varphi}}).$$ 
\end{lem}
\begin{proof}
	
For any $\varphi \in \mathfrak{X}_{\mathbb{I}}$, we denote by $\mathcal{P}=\mathrm{Ker}\ \varphi$ and by $\mathbb{I}_{\mathcal{P}}$ the localization of $\mathbb{I}$ at $\mathcal{P}$. Then $\mathbb{I}_{\mathcal{P}}$ is a discrete valuation ring with $\theta$ a fixed uniformizer of $\mathbb{I}_{\mathcal{P}}$. For the Galois representation 

$$\rho_{\mathcal{P}} : \mathrm{Gal}\left(\overline{\mathbf{Q}}/\mathbf{Q} \right) \stackrel{\rho_{\mathscr{F}}}{\rightarrow} \mathrm{GL}_2(\mathrm{Frac}(\mathbb{I}))=\mathrm{GL}_2(\mathrm{Frac}(\mathbb{I}_{\mathcal{P}}),$$ let $B$ be the $\mathbb{I}_{\mathcal{P}}$-submodule of $\mathrm{Frac}(\mathbb{I}_{\mathcal{P}})$  generated by $b(g)$ for all $g \in \mathrm{Gal}\left(\overline{\mathbf{Q}}/\mathbf{Q} \right)$. Since $\rho_{\mathscr{F}}$ is irreducible, $B \neq (0)$. Since $\rho_{\mathscr{F}}$ is continuous, by Definition \ref{lat} there exists a lattice $\mathcal{T} \subset \mathrm{Frac}(\mathbb{I})^{\oplus 2}$ which is stable under $\mathrm{Gal}(\overline{\mathbf{Q}}/\mathbf{Q})$-action. Then $\mathcal{T}_{\mathcal{P}}=\mathcal{T}\otimes_{\mathbb{I}} \mathbb{I}_{\mathcal{P}}$ is a stable lattice of $\rho_{\mathcal{P}}$ and $\mathrm{Im}\  \rho_{\mathcal{P}}$ is bounded. This implies $B=(\theta)^n$ for an integer $n$.

We replace $\rho_{\mathcal{P}}$ by $\begin{pmatrix} 1 & 0 \\ 0 & \theta^n \end{pmatrix}\rho_{\mathcal{P}}\begin{pmatrix} 1 & 0 \\ 0 & \theta^n \end{pmatrix}^{-1}$ and we denote by the same symbol $\rho_{\mathcal{P}}$ for this new Galois representation. Then $\mathrm{Im}\ \rho_{\mathcal{P}} \subset \mathrm{GL}_2\left(\mathbb{I}_{\mathcal{P}} \right)$ for new $\rho_{\mathcal{P}}$. We also denote by the same symbol $a(g), b(g), c(g), d(g)$ such that $\rho_{\mathcal{P}}(g)=\begin{pmatrix} a(g) & b(g) \\ c(g) & d(g) \end{pmatrix}$ for all $g \in \mathrm{Gal}\left(\overline{\mathbf{Q}}/\mathbf{Q} \right)$.

We denote by $\rho_{\varphi}$ the Galois representation $$\rho_{\varphi} : \mathrm{Gal}(\overline{\mathbf{Q}}/\mathbf{Q}) \stackrel{\rho_{\mathcal{P}}}{\rightarrow} \mathrm{GL}_2(\mathbb{I}_{\mathcal{P}}) \twoheadrightarrow \mathrm{GL}_2(\varphi(\mathbb{I}_{\mathcal{P}})), g \mapsto \begin{pmatrix} a_{\varphi}(g) & b_{\varphi}(g) \\ c_{\varphi}(g) & d_{\varphi}(g) \end{pmatrix}$$ and by $\rho_{f_{\varphi}}$ the Galois representation associated to $f_{\varphi}$. Since $\mathrm{tr}\rho_{\varphi}\left(\mathrm{Frob}_l \right)=\mathrm{tr}\rho_{f_{\varphi}}\left(\mathrm{Frob}_l \right)$ and $\mathrm{det}\rho_{\varphi}\left(\mathrm{Frob}_l \right)=\mathrm{det}\rho_{f_{\varphi}}\left(\mathrm{Frob}_l \right)$ for all primes $l \nmid Np$, we have $\rho_{\varphi}^{\mathrm{ss}} \cong \rho_{f_{\varphi}}^{\mathrm{ss}}$ by the Chebotarev density theorem. Since $\rho_{f_{\varphi}}$ is irreducible, so is $\rho_{\varphi}$. Thus $\rho_{\varphi} \cong \rho_{f_{\varphi}}$. Since $a_{\varphi}(g)=\varphi(a(g)), d_{\varphi}(g)=\varphi(d(g)), b_{\varphi}(g)c_{\varphi}(g)=\varphi(b(g)c(g))$ 
and $a_{\varphi}(g_0) \not\equiv d_{\varphi}(g_0)\ (\mathrm{mod}\ \varpi_{\varphi}$), we have $I\left(\rho_{f_{\varphi}} \right)=\varphi\left(I\left(\rho_{\mathscr{F}} \right) \right)$ by the definition of $I\left(\rho_{f_{\varphi}} \right)$. Thus the statement follows from Proposition \ref{2.8}. 

\end{proof}

\begin{lem}\label{4}
	Let us take the basis of $\mathrm{Frac}(\mathbb{I})^{\oplus 2}$ to be the same as the beginning of this section and let $\eta_1=\chi_1, \eta_2=\chi_2\kappa_{\mathrm{cyc}}\kappa$. Let $J$ be the ideal of $\mathbb{I}$ generated by $\mathrm{tr}\rho_{\mathscr{F}}(g)-\eta_1(g)-\eta_2(g)$ for all $g \in \mathrm{Gal}\left(\overline{\mathbf{Q}}/\mathbf{Q} \right)$ and $J^{\prime}$ the ideal generated by $a(g)-\eta_1(g)$ for all $g \in \mathrm{Gal}\left(\overline{\mathbf{Q}}/\mathbf{Q} \right)$. Then we have the following statements:
\begin{list}{}{}
\item[(1)] $I(\rho_{\mathscr{F}})=J=J^{\prime}$.
\item[(2)] Suppose $N=1$. Then $\rho_{\mathscr{F}}\left(\mathfrak{m} \right) \cong \overline{\mathbf{1}} \oplus \overline{\chi}$, where $\mathbf{1}$ is the trivial character. 
\end{list}
	 
\end{lem}

\begin{proof}
We first show $J=J^{\prime}$. Since $\mathrm{tr}\rho_{\mathscr{F}} \equiv \eta_1+\eta_2\ \mathrm{mod}\ J$, $a(g) \equiv \eta_1(g)\ \mathrm{mod}\ J$ or $a(g) \equiv \eta_2(g)\ \mathrm{mod}\ J$ for all $g \in \mathrm{Gal}\left(\overline{\mathbf{Q}}/\mathbf{Q} \right)$ by Lemma \ref{2.6}. Since the character $a\ \mathrm{mod}\ \mathfrak{m}=\overline{\chi}_1$ is unramified at $p$, we have $a(g) \equiv \eta_1(g)\ \mathrm{mod}\ J$ for all $g \in \mathrm{Gal}\left(\overline{\mathbf{Q}}/\mathbf{Q} \right)$. This implies $J^{\prime} \subset J$. We also have $J \subset J^{\prime}$ since 
$$\mathrm{tr}\rho_{\mathscr{F}}(g)-\eta_1(g)-\eta_2(g)=(a(g)-\eta_1(g))+(a(g^{-1})-\eta_1(g^{-1}))\eta_1\eta_2(g) \in J^{\prime}.$$

Now we prove $I(\rho_{\mathscr{F}})=J^{\prime}$. We have $b(g)c(g^{\prime}) \in J$ for any $g, g^{\prime} \in \mathrm{Gal}(\overline{\mathbf{Q}}/\mathbf{Q})$ by Lemma\ \ref{2.6} hence $I(\rho_{\mathscr{F}}) \subset J=J^{\prime}$. Let $K$ be the abelian extension of $\mathbf{Q}$ corresponding to $$\mathrm{Ker}\left[a\ \mathrm{mod}\ I\left(\rho_{\mathscr{F}} \right) : \mathrm{Gal}(\overline{\mathbf{Q}}/\mathbf{Q}) \longrightarrow \left(\mathbb{I}/I(\rho_{\mathscr{F}}) \right)^{\times},\ g \mapsto a(g)\ \mathrm{mod}\ I(\rho_{\mathscr{F}}) \right].$$ We denote by $\tilde{a} : \mathrm{Gal}\left(K/\mathbf{Q} \right) \hookrightarrow \left(\mathbb{I}/I(\rho_{\mathscr{F}}) \right)^{\times}$ the induced homomorphism. For all $g \in \mathrm{Gal}\left(\overline{\mathbf{Q}}/\mathbf{Q} \right)$, we denote by $\overline{g}$ the image of $g$ under $\mathrm{Gal}\left(\overline{\mathbf{Q}}/\mathbf{Q} \right) \twoheadrightarrow  \mathrm{Gal}\left(K/\mathbf{Q} \right)$. Write $a(g)=\eta_1(g)\left(1+m(g) \right)$ where $m(g) \in \mathfrak{m}$. Then $\tilde{a}(\overline{g})=\eta_1(g)\ \mathrm{mod}\ I(\rho_{\mathscr{F}})\cdot \left(1+m(g) \right)\ \mathrm{mod}\ I(\rho_{\mathscr{F}}).$ Note that $a\ \mathrm{mod}\ I\left(\rho_{\mathscr{F}} \right)$ is unramified outside $N$ by the equation (\ref{0118}), hence $K$ is a subfield of $\mathbf{Q}(\mu_N)$ by class field theory. On the other hand, the kernel of the map $\left(\mathbb{I}/I(\rho_{\mathscr{F}}) \right)^{\times} \rightarrow \left(\mathbb{I}/\mathfrak{m} \right)^{\times}$ is a pro-$p$ group, thus $\left(1+m(g) \right)\ \mathrm{mod}\ I(\rho_{\mathscr{F}})$ must be trivial under the assumption $p \nmid \phi(N)$. This implies $\eta_1(g) \equiv a(g)\ \mathrm{mod}\ I(\rho_{\mathscr{F}})$, hence $J^{\prime} \subset I(\rho_{\mathscr{F}})$. Specially when $N=1$, we have that $a\ \mathrm{mod}\ I\left(\rho_{\mathscr{F}} \right)$ is an unramified character. Thus $a\ \mathrm{mod}\ I\left(\rho_{\mathscr{F}} \right)$ is trivial by class field theory.

\end{proof}

Lemma \ref{4} tells us that $I(\rho_{\mathscr{F}})$ is a closed ideal in $\mathbb{I}$ under the $\mathfrak{m}$-adic topology.

\begin{prp}\label{5}
Let us take the basis of $\mathrm{Frac}(\mathbb{I})^{\oplus 2}$ to be the same as at the beginning of this section.  Let $L_{\infty}, L_{\infty}\left(Np \right)$ be the maximal unramified abelian $p$-extension of $\mathbf{Q}\left(\mu_{Np^{\infty}} \right)$ and the maximal abelian $p$-extension unramified outside $Np$ of $\mathbf{Q}\left(\mu_{Np^{\infty}} \right)$. We denote by $X_{\infty}=\mathrm{Gal}\left(L_{\infty}/\mathbf{Q}\left(\mu_{Np^{\infty}} \right) \right)$ and by $Y_{\infty}=\mathrm{Gal}\left(L_{\infty}\left(Np \right)/\mathbf{Q}\left(\mu_{Np^{\infty}} \right) \right)$ on which $\Delta_{Np}=\mathrm{Gal}\left(\mathbf{Q}\left(\mu_{Np^{\infty}} \right)/\mathbf{Q}_{\infty} \right)$ acts by conjugation. Then we have the following statements:

\begin{list}{}{}
\item[(1)]$\hat{G}_{\chi_1^{-1}\chi_2}(X)\mathbb{I} \subset I(\rho_{\mathscr{F}}).$
\item[(2)]Suppose the $\Lambda_{\chi_1\chi_2^{-1}}$-modules $X_{\infty}^{\chi_1\chi_2^{-1}}$ and $Y_{\infty}^{\chi_1^{-1}\chi_2}$ are cyclic. Then $I\left(\rho_{\mathscr{F}} \right)$ is principal. 
\end{list}

\end{prp}

\begin{proof}
Recall that $\kappa$ is the character $\mathrm{Gal}\left(\overline{\mathbf{Q}}/\mathbf{Q} \right) \twoheadrightarrow \mathrm{Gal}\left(\mathbf{Q}_{\infty}/\mathbf{Q} \right)  \stackrel{\sim}{\rightarrow} 1+p\mathbf{Z}_p \hookrightarrow \Lambda_{\chi}^{\times}$. Let $\eta_1=\chi_1$ and $\eta_2=\chi_2\kappa_{\mathrm{cyc}}\kappa$.  We have that $\eta_1(g) \equiv a(g)\ \mathrm{mod}\ I(\rho_{\mathscr{F}})$ and $\eta_2(g) \equiv d(g)\ \mathrm{mod}\ I(\rho_{\mathscr{F}})$ for all $g \in \mathrm{Gal}\left(\overline{\mathbf{Q}}/\mathbf{Q} \right)$ by Lemma \ref{2.6} and Lemma \ref{4}. We prove the proposition by using Wiles' construction (cf. \cite[Section 6]{Wi90}) of an uramified extension $N_{\infty}$ of $\mathbf{Q}(\mu_{Np^{\infty}})$. 

Let $B$ (resp. $C$) be an $\mathbb{I}$-submodule of $\mathrm{Frac}(\mathbb{I})$ generated by $b(g)$ (resp. $c(g)$) for all $g \in \mathrm{Gal}\left(\overline{\mathbf{Q}}/\mathbf{Q} \right)$. Since $c(g)B$ and $b(g)C$ are ideals of $\mathbb{I}$ for all $g \in \mathrm{Gal}\left(\overline{\mathbf{Q}}/\mathbf{Q} \right)$ by Lemma \ref{2.6}, $B$ and $C$ are finitely generated. We denote by $b$ the function $$b : \mathrm{Gal}\left(\overline{\mathbf{Q}}/\mathbf{Q} \right) \rightarrow B,\ g \mapsto b(g)$$ 
and we endow $B$ with the $\mathfrak{m}$-adic topology. 

\begin{cla}\label{0}
$b$ is continuous. 
	
\end{cla}

\begin{proof}
Since $\rho_{\mathscr{F}}$ is continuous, by Definition \ref{lat} there exists a lattice $\mathcal{T} \subset \mathrm{Frac}(\mathbb{I})^{\oplus 2}$ which is stable under $\mathrm{Gal}(\overline{\mathbf{Q}}/\mathbf{Q})$-action such that $\rho_{\mathscr{F}} : \mathrm{Gal}(\overline{\mathbf{Q}}/\mathbf{Q}) \longrightarrow \mathrm{Aut}_{\mathbb{I}}(\mathcal{T})$ is continuous with respect to the $\mathfrak{m}$-adic topology of $\mathrm{Aut}_{\mathbb{I}}(\mathcal{T})$. We denote by $V_i=\mathrm{Frac}(\mathbb{I})e_i$ and by $\mathcal{T}_i=\mathcal{T} \cap V_i\ \left(i=1, 2 \right).$ Then $\rho_{\mathscr{F}}\left(\mathcal{T}_i \right) \subset \mathcal{T}$. For any $xe_1 \in \mathcal{T}_1$ and $ye_2 \in \mathcal{T}_2$, we have $$\rho_{\mathscr{F}}\left(g \right)\left(xe_1 \right)=a\left(g \right)xe_1+c\left(g \right)xe_2,$$ $$\rho_{\mathscr{F}}\left(g \right)\left(ye_2 \right)=b\left(g \right)ye_1+d\left(g \right)ye_2.$$Since $a\left(g \right) \in \mathbb{I}$ by Lemma \ref{2.6}, $a\left(g \right)xe_1 \in \mathcal{T} \cap V_1=\mathcal{T}_1$ and $c\left(g \right)xe_2=\rho_{\mathscr{F}}\left(g \right)\left(xe_1 \right)-a\left(g \right)xe_1 \in \mathcal{T} \cap V_2=\mathcal{T}_2$. We also have $b\left(g \right)ye_1 \in \mathcal{T}_1$ by the same argument. This implies that $\mathcal{T}_1 \oplus \mathcal{T}_2$ is also a stable lattice of $\mathrm{Frac}\left(\mathbb{I} \right)^{\oplus 2}$. We replace $\mathcal{T}$ with $\mathcal{T}_1 \oplus \mathcal{T}_2$. The representation $\rho_{\mathscr{F}} : \mathrm{Gal}(\overline{\mathbf{Q}}/\mathbf{Q}) \longrightarrow \mathrm{Aut}_{\mathbb{I}}(\mathcal{T})$ is also continuous by the Artin-Rees lemma. We may regard $B$ as an $\mathbb{I}$-submodule of $\mathrm{Hom}_{\mathbb{I}}(\mathcal{T}_2, \mathcal{T}_1)$ via the injective homomorphism as follows: $$B \hookrightarrow \mathrm{Hom}_{\mathbb{I}}\left(\mathcal{T}_2, \mathcal{T}_1 \right),\ b(g) \mapsto b(g)(ye_2)=b(g)\cdot ye_1$$ for all $ye_2 \in \mathcal{T}_2$. Then $b$ is the following map: $$\mathrm{Gal}\left(\overline{\mathbf{Q}}/\mathbf{Q} \right) \stackrel{\rho_{\mathscr{F}}}{\rightarrow} \mathrm{Aut}_{\mathbb{I}}\left(\mathcal{T} \right) \rightarrow \mathrm{Hom}_{\mathbb{I}}(\mathcal{T}_2, \mathcal{T}_1).$$The homomorphism $\mathrm{Aut}_{\mathbb{I}}(\mathcal{T}) \rightarrow \mathrm{Hom}_{\mathbb{I}}(\mathcal{T}_2, \mathcal{T}_1)$ is continuous under the $\mathfrak{m}$-adic topology, hence $b$ is continuous. 
\end{proof}

Define $\overline{b}$ the following homomorphism:
$$\overline{b} : \mathrm{Gal}\left(\overline{\mathbf{Q}}/\mathbf{Q}\left(\mu_{Np^{\infty}} \right) \right) \stackrel{b}{\rightarrow} B \twoheadrightarrow B/I(\rho_{\mathscr{F}})B.$$Let $N_{\infty}$ be the abelian extension of $\mathbf{Q}\left(\mu_{Np^{\infty}} \right)$ corresponding to $\mathrm{Ker}\ \overline{b}$ and we denote by the same symbol $\overline{b}$ 

\begin{equation}\label{three}
\overline{b} : G=\mathrm{Gal}\left(N_{\infty}/\mathbf{Q}\left(\mu_{Np^{\infty}} \right) \right) \hookrightarrow B/I(\rho_{\mathscr{F}})B.
\end{equation}
For any $h \in \mathrm{Gal}\left(\overline{\mathbf{Q}}/\mathbf{Q} \right)$ and $g \in \mathrm{Gal}\left(\overline{\mathbf{Q}}/\mathbf{Q}\left(\mu_{Np^{\infty}} \right) \right)$, a matrix calculation shows that 
$$\overline{b}(hgh^{-1})=\eta_1\eta_2^{-1}(h)\overline{b}(g).$$
Let $\tilde{\gamma}$ be a topological generator of $\mathrm{Gal}\left(\mathbf{Q}\left(\mu_{Np^{\infty}} \right)/\mathbf{Q}\left(\mu_{Np} \right) \right)$ which is sent to $\gamma$ under the canonical isomorphism $\mathrm{Gal}\left(\mathbf{Q}\left(\mu_{Np^{\infty}} \right)/\mathbf{Q}\left(\mu_{Np} \right) \right) \rightarrow \mathrm{Gal}\left(\mathbf{Q}_{\infty}/\mathbf{Q} \right)$. The above arguments tell us that $\overline{b}\left(G \right)$ is a $\Lambda_{\chi_1\chi_2^{-1}}=\mathbf{Z}_p[\chi_1\chi_2^{-1}][[X]]$-module under the surjection
\begin{equation*}
\mathbf{Z}_p[[\mathrm{Gal}\left(\mathbf{Q}\left(\mu_{Np^{\infty}} \right)/\mathbf{Q} \right)]] \stackrel{\eta_1\eta_2^{-1}}{\longrightarrow} \Lambda_{\chi_1\chi_2^{-1}},\ u^{-1}\tilde{\gamma}^{-1} \mapsto 1+X
\end{equation*}
and $\mathrm{Gal}\left(\mathbf{Q}(\mu_{Np^{\infty}})/\mathbf{Q}_{\infty} \right)$ acts on $\overline{b}\left(G \right)$ via $\chi_1\chi_2^{-1}$. 
\begin{cla}\label{d1}
The canonical homomorphism $\overline{b}\left(G \right) \displaystyle\otimes_{\Lambda_{\chi_1\chi_2^{-1}}} \mathbb{I} \rightarrow B/I(\rho_{\mathscr{F}})B$ is an isomorphism. 
\end{cla}
\begin{proof}
The injectivity follows from the assumption that $\mathbb{I}$ is flat over $\Lambda_{\chi_1\chi_2^{-1}}$ by applying the base extension $\displaystyle\otimes_{\Lambda_{\chi_1\chi_2^{-1}}} \mathbb{I}$ to the equation (\ref{three}). For any $g \in \mathrm{Gal}\left(\overline{\mathbf{Q}}/\mathbf{Q} \right)$, consider the commutator $[g, g_0] \in \mathrm{Gal}\left(\overline{\mathbf{Q}}/\mathbf{Q}\left(\mu_{N^{\infty}p^{\infty}} \right) \right)$ we have $$\overline{b}([g, g_0])=\dfrac{\lambda-1}{\lambda}\eta_2(g)^{-1}\overline{b}(g),$$
where $\lambda=\varepsilon_2(g_0)$. Since $\lambda \not\equiv 1\ \mathrm{mod}\ \mathfrak{m}$, we have $\overline{b}\left([g, g_0] \right)\otimes \eta_2(g)\dfrac{\lambda}{\lambda-1} \in \overline{b}\left(G \right) \otimes_{\Lambda_{\chi_1\chi_2^{-1}}} \mathbb{I}$. This completes the proof of Claim \ref{d1}. 
\end{proof}
$N_{\infty}/\mathbf{Q}(\mu_{Np^{\infty}})$ is unramified at $p$ by the equation (1). Since the conductor of $\chi_1\chi_2^{-1}$ is $Np$ under the condition (D), $N_{\infty}/\mathbf{Q}(\mu_{Np^{\infty}})$ is also unramified at the primes dividing $N$ by class field theory (see the proof of \cite[Lemma 6.1]{Wi90}). Thus $N_{\infty}/\mathbf{Q}(\mu_{Np^{\infty}})$ is everywhere unramified.

We fix the Iwasawa-Serre isomorphism as follows:
\begin{equation}\label{IS}
\mathbf{Z}_p[\chi_1\chi_2^{-1}][[\mathrm{Gal}\left(\mathbf{Q}\left(\mu_{Np^{\infty}} \right)/\mathbf{Q}\left(\mu_{Np} \right) \right)]] \stackrel{\sim}{\rightarrow} \Lambda_{\chi_1\chi_2^{-1}}, \tilde{\gamma} \mapsto u^{-1}(1+X)^{-1}.
\end{equation}
Then we have the following $\mathbb{I}$-homomorphisms:
\begin{equation}\label{y1}
X_{\infty}^{\chi_1\chi_2^{-1}}\otimes_{\Lambda_{\chi_1\chi_2^{-1}}}\mathbb{I} \twoheadrightarrow \overline{b}\left(G \right)\otimes_{\Lambda_{\chi_1\chi_2^{-1}}}\mathbb{I} \stackrel{\sim}{\rightarrow} B/I(\rho_{\mathscr{F}})B.
\end{equation}
By taking the Fitting ideal, we have the inclusion relation as follows: $$\mathrm{Fitt}_{\Lambda_{\chi_1\chi_2^{-1}}}(X_{\infty}^{\chi_1\chi_2^{-1}})\mathbb{I}=\mathrm{Fitt}_{\mathbb{I}}(X_{\infty}^{\chi_1\chi_2^{-1}}\otimes_{\Lambda_{\chi_1\chi_2^{-1}}}\mathbb{I}) \subset \mathrm{Fitt}_{\mathbb{I}}(B/I(\rho_{\mathscr{F}})B) \subset I(\rho_{\mathscr{F}}).$$By the Iwasawa main conjecture (Theorem of Mazur-Wiles) we have $$\mathrm{Fitt}_{\Lambda_{\chi_1\chi_2^{-1}}}(X_{\infty}^{\chi_1\chi_2^{-1}})=G_{\chi_1^{-1}\chi_2\omega}(u^2(1+X)-1)\Lambda_{\chi_1\chi_2^{-1}}=\hat{G}_{\chi_1^{-1}\chi_2}(X)\Lambda_{\chi_1\chi_2^{-1}}.$$ Thus $\hat{G}_{\chi_1^{-1}\chi_2}(X)\mathbb{I} \subset I(\rho_{\mathscr{F}})$. This completes the proof of (1) of the proposition. 

Similarly, we denote by $M_{\infty}\left(Np \right)$ the abelian extension of $\mathbf{Q}\left(\mu_{Np^{\infty}} \right)$ corresponding to $$\mathrm{Ker} \left(\overline{c} : \mathrm{Gal}\left(\overline{\mathbf{Q}}/\mathbf{Q}\left(\mu_{Np^{\infty}} \right) \right) \rightarrow C/I\left(\rho_{\mathscr{F}} \right)C, \ g \mapsto \overline{c}(g) \right)$$ and by $H=\mathrm{Gal}\left(M_{\infty}\left(Np \right)/\mathbf{Q}\left(\mu_{Np^{\infty}} \right) \right)$. Then $\overline{c}\left(H \right)$ is a $\Lambda_{\chi_1^{-1}\chi_2}$-module under the surjection

\begin{equation*}
\mathbf{Z}_p[[\mathrm{Gal}\left(\mathbf{Q}\left(\mu_{Np^{\infty}} \right)/\mathbf{Q} \right)]] \stackrel{\eta_1^{-1}\eta_2}{\longrightarrow} \Lambda_{\chi_1\chi_2^{-1}},\ u^{-1}\tilde{\gamma} \mapsto 1+X
\end{equation*}
and the map $\overline{c}\left(H \right) \displaystyle\otimes_{\Lambda_{\chi_1^{-1}\chi_2}} \mathbb{I} \rightarrow C/I(\rho_{\mathscr{F}})C$ induced by $\overline{c}$ is an isomorphism by the same arguments as in Claim \ref{d1}. Hence we have the surjective homomorphism as follows: 
\begin{equation}\label{v1}
Y_{\infty}^{\chi_1^{-1}\chi_2}\displaystyle\otimes_{\Lambda_{\chi_1^{-1}\chi_2}}\mathbb{I} \twoheadrightarrow C/I(\rho_{\mathscr{F}})C.
\end{equation}
Note that in the equation (\ref{v1}), we endowed $Y_{\infty}^{\chi_1^{-1}\chi_2}$ with the $\Lambda_{\chi_1\chi_2^{-1}}$-module structure under the isomorphism as follows:
$$\mathbf{Z}_p[\chi_1^{-1}\chi_2][[\mathrm{Gal}\left(\mathbf{Q}\left(\mu_{Np^{\infty}} \right)/\mathbf{Q}\left(\mu_{Np} \right) \right)]] \stackrel{\sim}{\rightarrow} \Lambda_{\chi_1^{-1}\chi_2}, \tilde{\gamma} \mapsto u(1+X).$$

By the equations (\ref{y1}) and (\ref{v1}), there exists a $g_B \in X_{\infty}$ (resp. $g_C \in Y_{\infty}$) such that $B/I(\rho_{\mathscr{F}})B$ (resp. $C/I(\rho_{\mathscr{F}})C$) is generated by $\overline{b}(g_B)$ (resp. $\overline{c}(g_C)$).  By Nakayama's lemma, $B$ (resp. $C$) is generated by $b(g_B)$ (resp. $c(g_C)$) over $\mathbb{I}$. This implies $I\left(\rho_{\mathscr{F}} \right)=BC=\left(b\left(g_B \right)c\left(g_C \right) \right)$.

\end{proof}

Define the Eisenstein ideal $I(\chi, \mathbb{I})$ the ideal of $\mathbf{T}(\chi, \mathbb{I})$ which is generated by $T(l)-1-\eta(\mathrm{Frob}_{l})$ for all primes $l \neq p$ and $T(p)-1$.

\begin{cor}\label{m}
Let the assumptions and the notations be as in Theorem \ref{yandong}. Assume the condition (R). We have $I\left(\rho_{\mathscr{F}} \right)=\hat{G}_{\chi}\left(X \right)\mathbb{I}$.
\end{cor}

\begin{proof}
Since $\mathbf{T}(\chi, \mathbb{I})=\mathbf{T}(\chi, \Lambda_{\chi}) \otimes_{\Lambda_{\chi}} \mathbb{I}$, $\mathbf{T}(\chi, \mathbb{I}) $ is isomorphic to $\mathbb{I}$ by the assumption (R). Since $N=1$, $I\left(\rho_{\mathscr{F}}  \right)$ is generated by $c(l, \mathscr{F})-1-\eta(\mathrm{Frob}_{l})$ for all primes $l \neq p$ by Lemma \ref{4} and the Chebotarev density theorem. We also have $c(p, \mathscr{F})-1=\varepsilon_1(\mathrm{Frob}_p)-1=a(\mathrm{Frob}_p)-1 \in I\left(\rho_{\mathscr{F}} \right)$ by Theorem \ref{w} and Lemma \ref{4}. Thus the canonical isomorphism $\mathbb{I} \rightarrow \mathbf{T}(\chi, \mathbb{I})$ sends $I\left(\rho_{\mathscr{F}} \right)$ to the Eisenstein ideal $I(\chi, \mathbb{I})$. On the other hand, the canonical homomorphism $$\mathbb{I}/\hat{G}_{\chi}(X) \mathbb{I} \rightarrow \mathbf{T}(\chi, \mathbb{I})/\left(I(\chi, \mathbb{I}), \hat{G}_{\chi}(X) \right)$$ is an isomorphism by \cite[Theorem 4.1]{Wi90}. This implies $I\left(\rho_{\mathscr{F}} \right) \subset \hat{G}_{\chi}\left(X \right)\mathbb{I}$. Hence they must be coincide by Proposition \ref{5}.
	
\end{proof}

The next corollary is obviously deduced from (2) of Proposition \ref{5}.

\begin{cor}\label{r}
Let the assumptions and the notations be as in Theorem \ref{yandong}. Assume the conditions (C) and (P). We have $I\left(\rho_{\mathscr{F}} \right)=\hat{G}_{\chi}\left(X \right)\mathbb{I}$.
	
\end{cor}

Now we prove Theorem \ref{yandong}. 

\begin{pf}
For any $\varphi \in \mathfrak{X}_{\mathbb{I}}$, let $\varpi_{\varphi}$ be a fixed uniformizer of $\varphi({\mathbb{I}})$. Then

$$\mathrm{ord}_{\varpi_{\varphi}}(\varphi\left(I\left(\rho_{\mathscr{F}} \right) \right) \leq \mathrm{ord}_{\varpi_{\varphi}}\left(\hat{G}_{\chi_1^{-1}\chi_2}(\zeta_{\varphi}u^{k_{\varphi}-2}-1) \right)$$by (1) of Proposition \ref{5}. Since $\chi_1 \neq \chi_2\omega$, the character $\chi_{\zeta_{\varphi}}\chi_1^{-1}\chi_2\omega$ does not factor through $\mathrm{Gal}\left(\mathbf{Q}_{\infty}/\mathbf{Q} \right)$. Thus $$L_p(1-k_{\varphi}, \chi_{\zeta_{\varphi}}\chi_1^{-1}\chi_2\omega)=\hat{G}_{\chi_1^{-1}\chi_2}(\zeta_{\varphi}u^{k_{\varphi}-2}-1)$$ by Section 2.3. Combine Proposition \ref{2.8} and Lemma \ref{1} we have $$\mathrm{ord}_{\varpi_{\varphi}}(\varphi\left(I\left(\rho_{\mathscr{F}} \right) \right))+1 =\sharp \mathcal{L}(\rho_{f_{\varphi}}),$$then $$\sharp \mathcal{L}(\rho_{f_{\varphi}}) \leq \mathrm{ord}_{\varpi_{\varphi}}\left(L_p(1-k_{\varphi}, \chi_{\zeta_{\varphi}}\chi_1^{-1}\chi_2\omega) \right)+1.$$

If we assume the condition (R) or both of the conditions (C) and (P), we have $I\left(\rho_{\mathscr{F}} \right)=\hat{G}_{\chi_1^{-1}\chi_2}\left(X \right)\mathbb{I}$ by Corollary \ref{m} and Corollary \ref{r}. Then $$\sharp \mathcal{L}(\rho_{f_{\varphi}})=\mathrm{ord}_{\varpi_{\varphi}}\left(L_p(1-k_{\varphi}, \chi_{\zeta_{\varphi}}\chi_1^{-1}\chi_2\omega) \right)+1.$$ Specially when (R) satisfied, $\chi_1=1$ and $\chi_2=\chi$ by Lemma \ref{4}. This completes the proof of Theorem \ref{yandong}.

\end{pf}

\subsection{Discussion of the variation of $\sharp\mathcal{L}\left(\rho_{f_{\varphi}} \right)$ by means of $L_p$}

We use the following lemma to prove Corollary \ref{yandong2}. 
\begin{lem}\label{L}
Let $\mathcal{O}$ be the ring of integers of a finite extension of $\mathbf{Q}_p$ and $F(X) \in \mathcal{O}[X]$ a distinguished polynomial. Then there exists an integer $r \in \mathbf{Z}_{\geq 0}$ such that for any $(k, \zeta) \in \mathbf{Z}_{\geq 0} \times \left(\mu_{p^{\infty}} \setminus \mu_{p^r} \right)$, $$\mathrm{ord}_p\left(F\left(\zeta u^k-1 \right) \right)=\dfrac{\mathrm{deg}F(X)}{\left(p-1 \right)p^{r_{\zeta}-1}},$$ where $\zeta$ is a primitive $p^{r_{\zeta}}$-th root of unity.

\end{lem}
\begin{proof}
Decompose $$F(X)=\displaystyle\prod^n_{i=1}(X-\alpha_i)$$ and choose an integer $r \geq 0$ such that $\mathrm{ord}_p(\alpha_i) > \dfrac{1}{(p-1)p^{r -1}}$ for any $\alpha_i$. Then for any $(k, \zeta) \in \mathbf{Z}_{\geq 0} \times \left(\mu_{p^{\infty}} \setminus \mu_{p^r} \right)$, we have 

\begin{align*}
\mathrm{ord}_p(\zeta u^k-1-\alpha_i)&=\mathrm{ord}_p\left(\zeta\left(\mathrm{exp}\left(k\cdot\mathrm{log}(u)\right)-1\right)+(\zeta -1)-\alpha_i \right)\\
&=\dfrac{1}{(p-1)p^{r_{\zeta}-1}},
\end{align*}
where exp and log are the $p$-adic exponential and logarithm functions. Thus $$\mathrm{ord}_p\left(F\left(\zeta u^k-1 \right) \right)=\displaystyle\sum_{i=1}^n \mathrm{ord}_p(\zeta u^k-1-\alpha_i)=\dfrac{\mathrm{deg}F(X)}{\left(p-1 \right)p^{r_{\zeta}-1}}.$$

\end{proof}

Let us return to the proof of Corollary \ref{yandong2}. 

\begin{pf2}
For (1) it is sufficient to show that there exists an $r \in \mathbf{Z}_{\geq 0}$ such that for any $\varphi \in \mathfrak{X}_{\mathbb{I}}^{\left(r \right)}$, $$\mathrm{ord}_{\varpi_{\varphi}}(L_p(1-k_{\varphi}, \chi_{\zeta_{\varphi}}\chi_1^{-1}\chi_2\omega)) \leq \mathrm{rank}_{\Lambda_{\chi}}\mathbb{I}\cdot\mathrm{deg}\hat{G}_{\chi_1^{-1}\chi_2}^{*}(X).$$ 
Since $\hat{G}_{\chi_1^{-1}\chi_2}(X)$ is not divisible by a uniformizer of $\mathbf{Z}_p[\chi_1^{-1}\chi_2]$ by the Ferrero-Washington's theorem \cite{FW77}, the Weierstrass preparation theorem enables one to decompose

$$\hat{G}_{\chi_1^{-1}\chi_2}(X)=\hat{G}_{\chi_1^{-1}\chi_2}^{*}(X) U(X),$$where $\hat{G}_{\chi_1^{-1}\chi_2}^{*}(X)$ is a distinguished polynomial and $U(X)$ a unit in $\Lambda_{\chi_1^{-1}\chi_2}$. We apply Lemma \ref{L} to $\hat{G}_{\chi_1^{-1}\chi_2}^{*}(X)$ . Then there exists an $r \in \mathbf{Z}_{\geq 0}$ such that for any $\varphi \in \mathfrak{X}_{\mathbb{I}}^{(r)}$,

\begin{align}\label{10}
\mathrm{ord}_p\left(\varphi\left(\hat{G}_{\chi_1^{-1}\chi_2}^{*}(X) \right) \right)=\dfrac{\mathrm{deg}\hat{G}_{\chi_1^{-1}\chi_2}^{*}(X)}{\left(p-1 \right)p^{r_{\varphi}-1}},
\end{align}
where $\varphi \mid_{\Lambda_{\chi}}=\nu_{k_{\varphi}, \zeta_{\varphi}}$ and $\zeta_{\varphi}$ is a primitive $p^{r_{\varphi}}$-th root of unity such that $r_{\varphi} > r$.

Let us take a $\varphi \in \mathfrak{X}_{\mathbb{I}}^{(r)}$. For the extension of the discrete valuation rings $\varphi(\mathbb{I}) \supset \mathbf{Z}_p[\chi][\zeta_{\varphi}]$, since $[\mathrm{Frac}(\varphi(\mathbb{I})) : \mathrm{Frac}(\mathbf{Z}_p[\chi][\zeta_{\varphi}])] \leq \mathrm{rank}_{\Lambda_{\chi}}\mathbb{I}$, so is the ramification index $e_{\varphi}$. Since $r_{\varphi} > 0$, the ramification index in the extension $\mathbf{Z}_p[\chi][\zeta_{\varphi}] \supset \mathbf{Z}_p$ is $\left(p-1 \right)p^{r_{\varphi}-1}$. Then by the equation (\ref{10}), we have

\begin{align*}
\mathrm{ord}_{\varpi_{\varphi}}(L_p(1-k_{\varphi}, \chi_{\zeta_{\varphi}}\chi_1^{-1}\chi_2\omega))&=\mathrm{ord}_{\varpi_{\varphi}}\left(\varphi\left(\hat{G}_{\chi_1^{-1}\chi_2}^{*}(X) \right) \right)\\
&=e_{\varphi}(p-1)p^{r_{\varphi}-1}\dfrac{\mathrm{deg}\hat{G}_{\chi_1^{-1}\chi_2}^{*}(X)}{\left(p-1 \right)p^{r_{\varphi}-1}} \\
&=e_{\varphi}\mathrm{deg}\hat{G}_{\chi_1^{-1}\chi_2}^{*}(X)\\
&\leq \mathrm{rank}_{\Lambda_{\chi}}\mathbb{I}\cdot \mathrm{deg}\hat{G}^{*}_{\chi_1^{-1}\chi_2}(X).
\end{align*}
This completes the proof of (1) of Corollary \ref{yandong2} and (2) is easily deduced from (1). 

Now we assume that $\mathbb{I}$ is isomorphic to $\mathcal{O}[[X]]$ with $\mathcal{O}$ the ring of integers of a finite extension $K$ of $\mathbf{Q}_p$. We choose a uniformizer $\varpi$ of $\mathcal{O}$. Let $f_1(X), \cdots, f_m(X)$ be the generators of $I\left(\rho_{\mathscr{F}} \right)$. For each $i=1, \cdots, m$, decompose $$f_i(X)=\varpi^{\mu_i}P_i(X)U_i(X),$$ where $P_i(X)$ is a distinguished polynomial and $U_i(X)$ a unit in $\mathcal{O}[[X]]$. Let $$F(X)=\displaystyle\prod_{i=1}^m P_i(X).$$ We apply Lemma \ref{L} to $F(X)$. Then there exists an $r_1 \in \mathbf{Z}_{\geq 0}$ such that for any $\varphi \in \mathfrak{X}_{\mathbb{I}}^{(r_1)}$, 
\begin{align}\label{r_0}
\mathrm{ord}_p\varphi\left(f_i(X) \right)=\mu_i\mathrm{ord}_p\varpi+\dfrac{\mathrm{deg}P_i(X)}{(p-1)p^{r_{\varphi}-1}}.
\end{align}

We denote by $\zeta_m$ a primitive $m$-th root of unity for a positive integer $m$. Let $r_2$ be the largest integer $j$ such that $\zeta_{p^j} \in K$ and let us take a $\varphi \in \mathfrak{X}_{\mathbb{I}}^{(r_2)}$. Write $K_{\varphi}=\mathrm{Frac}\left(\varphi\left(\mathbb{I} \right) \right)$ for short. First we assume $r_2 > 0$. Then we have $K \cap \mathbf{Q}_p\left(\zeta_{\varphi} \right)=\mathbf{Q}_p\left(\zeta_{p^{r_2}} \right)$ and $\mathrm{Gal}\left(K_{\varphi}/\mathbf{Q}_p\left(\zeta_{\varphi} \right) \right) \stackrel{\sim}{\rightarrow} \mathrm{Gal}\left(K/\mathbf{Q}_p\left(\zeta_{p^{r_2}} \right) \right).$ Since $\mathbb{I}$ is isomorphic to $\mathcal{O}[[X]]$, the residue degree in $K_{\varphi}/\mathbf{Q}_p\left(\zeta_{\varphi} \right)$ and $K/\mathbf{Q}_p\left(\zeta_{p^{r_2}} \right)$ coincide, so are the ramification index. Hence the ramification index of $K_{\varphi}$ over $\mathbf{Q}_p$ is $e(p-1)p^{r_{\varphi}-1}$, where $e$ is the ramification index of $K$ over $\mathbf{Q}_p(\zeta_{p^{r_2}})$. If $r_2=0$, we may enlarge $\mathcal{O}$ to $\mathcal{O}^{\prime}=\mathcal{O}[\zeta_p]$ since $\mathcal{O}[\zeta_{\varphi}]=\mathcal{O}^{\prime}[\zeta_{\varphi}]$ for $\varphi \in \mathfrak{X}_{\mathbb{I}}^{(r_2)}$. Then the argument above also holds, i.e. there exists a constant $e$ such that the ramification index of $K_{\varphi}$ over $\mathbf{Q}_p$ is $e(p-1)p^{r_{\varphi}-1}$. Note that $e$ is the ramification index of $K\left(\zeta_p \right)/\mathbf{Q}_p\left(\zeta_p \right)$ if $r_2=0$.

Since $\hat{G}_{\chi_1^{-1}\chi_2}(X)\mathbb{I} \subset I\left(\rho_{\mathscr{F}} \right)$ and $\varpi \nmid \hat{G}_{\chi_1^{-1}\chi_2}(X)$, we have that $$\mathcal{Z}=\set{ i=1, \cdots, m | \mu_i=0}$$ is nonempty. Let $$l=\mathrm{min}\set{ \mathrm{deg}P_i(X) | i \in \mathcal{Z} }$$ and let us take an $r_3 \in \mathbf{Z}_{\geq 0}$ such that for any $i \notin \mathcal{Z}$, 
\begin{align}\label{r2}
(p-1)p^{r_3-1}\mu_i\mathrm{ord}_p\varpi+\mathrm{deg}P_i(X) \geq l.
\end{align}
Let $r^{\prime}=\mathrm{max}\set{r_1, r_2, r_3}$ and let us take a $\varphi \in \mathfrak{X}_{\mathbb{I}}^{(r^{\prime})}$. Then we have 

\begin{align}\label{s1}
\sharp\mathcal{L}(\rho_{f_{\varphi}})=\mathrm{min}\set{\mathrm{ord}_{\varpi_{\varphi}}\varphi\left(f_i(X) \right) | 1 \leq i \leq n}+1.
\end{align}
Since the ramification index of $K_{\varphi}$ over $\mathbf{Q}_p$ is $e(p-1)p^{r_{\varphi}-1}$, we have

\begin{align}\label{s2}
\mathrm{ord}_{\varpi_{\varphi}}\varphi\left(f_i(X) \right)=e(p-1)p^{r_{\varphi}-1}\mu_i\mathrm{ord}_p\varpi+e\cdot\mathrm{deg}P_i(X)
\end{align}
for each $1 \leq i \leq n$ by the equation (\ref{r_0}). Thus
\begin{align}\label{s3}
\mathrm{min}\set{\mathrm{ord}_{\varpi_{\varphi}}\varphi\left(f_i(X) \right) | 1 \leq i \leq n}=el
\end{align}
by the equation (\ref{r2}). 
Combine the equation (\ref{s1}) and (\ref{s3}), we have that $\sharp\mathcal{L}(\rho_{f_{\varphi}})=el+1$ is constant. This completes the proof of (3) of Corollary \ref{yandong2}.

Now we assume the condition (R) or both of the conditions (C) and (P). We have $\sharp\mathcal{L}(\rho_{f_{\varphi}})=\mathrm{ord}_{\varpi_{\varphi}}\left(L_p\left(1-k_{\varphi}, \chi_{\zeta_{\varphi}}\chi_1^{-1}\chi_2\omega \right) \right)+1$ by (2) and (3) of Theorem \ref{yandong}. 
We fix a $\zeta \in \mu_{p^{\infty}}$. First we assume that $L_p(1-s, \chi_{\zeta}\chi_1^{-1}\chi_2\omega)$ has a zero $s_0 \in \mathbf{Z}_p$. Let $\left\{k_n \right\}$ be the sequence defined as follows: 

\begin{list}{}{}
\item[(i)]$k_n=s_0+p^n$ if $s_0 \in \mathbf{Z}$,
\item[(ii)]$k_n=\displaystyle \sum^{n}_{i=0}a_i p^i$ if $s_0=\displaystyle \sum^{\infty}_{i=0}a_i p^i$ such that $0 \leq a_i \leq p-1$ and $s_0 \not\in \mathbf{Z}$.
\end{list}
Then $\sharp \mathcal{L}(\rho_{f_{\varphi}})$ is unbounded when $k_{\varphi}$ runs over the sequence $\bigl\{k_n \bigr\}$.

Now we suppose that $L_p \left(s, \chi_{\zeta}\chi_1^{-1}\chi_2\omega \right)$ has no zero in $\mathbf{Z}_p$ and we prove that $\mathrm{ord}_{\varpi_{\varphi}}\left(L_p\left(1-k_{\varphi}, \chi_{\zeta}\chi_1^{-1}\chi_2\omega \right) \right)$ is bounded by contradiction. Suppose that $\mathrm{ord}_{\varpi_{\varphi}}\left(L_p\left(1-k_{\varphi}, \chi_{\zeta}\chi_1^{-1}\chi_2\omega \right) \right)$ is unbounded. Then there exists a sequence $\left\{k_n \right\}$ such that $k_n \geq 2$ and $$\lim_{n \to \infty}L_p \left(1-k_n, \chi_{\zeta}\chi_1^{-1}\chi_2\omega \right)=0.$$Since $\mathbf{Z}_p$ is compact and $L_p$ is a continuous function, $L_p \left(s, \chi_{\zeta}\chi_1^{-1}\chi_2\omega \right)$ must have zero in $\mathbf{Z}_p$ which contradicts to our assumption. Hence $\sharp\mathcal{L}(\rho_{f_{\varphi}})$ is bounded. This completes the proof of (4) of Corollary \ref{yandong2}. 

\end{pf2}

\subsection{Proof of Corollary \ref{yandong3}}
We denote by $\mathbb{F}$ the residue field $\mathbb{I}/\mathfrak{m}$. The following lemma is a generalization of the arguments in \cite[Appendix I ]{Maz} for more general settings.
\begin{lem}\label{Maz}
Let the assumptions and the notations be as in Theoren \ref{yandong}. Assume the conditions (D), (C), (P) and (F). Let $\mathcal{T}$ be a stable lattice which is free over $\mathbb{I}$. Then $\mathcal{T}\otimes_{\mathbb{I}} \varphi(\mathbb{I})$ is a mod $\varpi_{\varphi}$ not semi-simple lattice for any $\varphi \in \mathfrak{X}_{\mathbb{I}}$.

\end{lem}

\begin{proof}

We have $I\left(\rho_{\mathscr{F}} \right)=\hat{G}_{\chi_1^{-1}\chi_2}(X)\mathbb{I}$ under the conditions (D), (C) and (P) by Corollary \ref{r}. Let us take a stable lattice $\mathcal{T} \cong \mathbb{I}^{\oplus 2}$ and we consider the following representation:
$$\rho=\rho_{\mathscr{F}, \mathcal{T}} : \mathrm{Gal}\left(\overline{\mathbf{Q}}/\mathbf{Q} \right) \rightarrow \mathrm{GL}_2\left(\mathbb{I} \right).$$
Write $\mathfrak{L}=\hat{G}_{\chi_1^{-1}\chi_2}(X)\mathbb{I}$ for short. The condition (P) enables us to define $\mathrm{Frac}\left(\mathbb{I}/\mathfrak{L} \right)$ and $\mathrm{Frac}\left(\mathbb{I}/\mathfrak{L} \right)$ is of characteristic zero by the Ferrero-Washington theorem. We denote by $\rho\ \mathrm{mod}\ \mathfrak{L}$ the representation as follows: $$\rho\ \mathrm{mod}\ \mathfrak{L} : \mathrm{Gal}\left(\overline{\mathbf{Q}}/\mathbf{Q} \right) \stackrel{\rho}{\rightarrow} \mathrm{GL}_2\left(\mathbb{I} \right) \stackrel{\mathrm{mod}\ \mathfrak{L}}{\longrightarrow} \mathrm{GL}_2\left(\mathbb{I}/\mathfrak{L} \right).$$ Since $\mathrm{tr}\rho\ \mathrm{mod}\ \mathfrak{L}$ is the sum of two characters, we have $\rho\ \mathrm{mod}\ \mathfrak{L}$ is reducible by the Brauer-Nesbitt theorem. Let $\left\{v_1, v_2 \right\}$ be a basis corresponding to $\rho\ \mathrm{mod}\ \mathfrak{L}$ such that $\left(\mathbb{I}/\mathfrak{L} \right)v_1$ is stable under $\rho\ \mathrm{mod}\ \mathfrak{L}$. Let $\tilde{v}_i \in \mathcal{T}$ be a lift of $v_i$ $(i=1, 2)$. Since $\mathbb{I}$ is complete under the $\mathfrak{m}$-adic topology, $\mathcal{T}$ is generated by $\tilde{v}_1$ and $\tilde{v}_2$ over $\mathbb{I}$. Since $\bigcap_{n} \mathfrak{m}^n=(0)$, we have $\mathcal{T}=\mathbb{I}\tilde{v}_1 \oplus \mathbb{I}\tilde{v}_2$. Thus $\mathcal{T}^{\prime}=\mathbb{I}\tilde{v}_1 \oplus \mathfrak{L}\tilde{v}_2$ is also a stable $\mathbb{I}$-free lattice and $\mathcal{T}/\mathcal{T}^{\prime} \cong \mathbb{I}/\hat{G}_{\chi_1^{-1}\chi_2}(X)\mathbb{I}.$

For any $\varphi \in \mathfrak{X}_{\mathbb{I}}$, we denote by $T, T^{\prime}$ the lattices $\mathcal{T} \otimes_{\mathbb{I}} \varphi(\mathbb{I}), \mathcal{T}^{\prime} \otimes_{\mathbb{I}} \varphi(\mathbb{I})$. Let $\sharp\mathcal{L}(\rho_{f_{\varphi}})=n+1$. Since $I\left(\rho_{\mathscr{F}} \right)=\hat{G}_{\chi_1^{-1}\chi_2}(X)\mathbb{I}$, we have 
$$T/T^{\prime}=\dfrac{\mathcal{T} \otimes_{\mathbb{I}} \varphi(\mathbb{I})} {\mathcal{T}^{\prime} \otimes_{\mathbb{I}} \varphi(\mathbb{I})}\cong \left(\mathcal{T}/\mathcal{T}^{\prime} \right) \otimes_{\mathbb{I}} \varphi(\mathbb{I}) \cong \left(\mathbb{I}/\hat{G}_{\chi_1^{-1}\chi_2}(X)\mathbb{I} \right) \otimes_{\mathbb{I}} \varphi(\mathbb{I}) \cong \varphi(\mathbb{I})/\varpi_{\varphi}^n.$$Thus $d\left([T], [T^{\prime}] \right)=n$. Since $\mathcal{L}\left(\rho_{f_{\varphi}} \right)$ is a segment by (6) of Proposition \ref{2.c} and Proposition \ref{2.8}, $[T]$ has exactly one neighbor in $\mathcal{L}\left(\rho_{f_{\varphi}} \right)$. Thus $T$ is a mod $\varpi_{\varphi}$ not semi-simple lattice by (3) of Proposition \ref{2.c}.

\end{proof}

Under the above preparation, we return to the proof of Corollary \ref{yandong3}.
\begin{pf3}
Let us take an $m \in \mathbf{Z}_{>0}$. Let $\zeta \in \mu_{p^{\infty}}$ such that $L_p \left(1-s, \chi_{\zeta}\chi_1^{-1}\chi_2\omega \right)$ has a zero in $\mathbf{Z}_p$. Then Corollary \ref{yandong2} (4) tells us that there exists a $\varphi \in \mathfrak{X}_{\mathbb{I}, \zeta}$ such that $\sharp\mathcal{L}(\rho_{f_{\varphi}})=n+1 > m$.

Now we fix such $\varphi$ and we denote by $\mathcal{P}=\mathrm{Ker}\ \varphi$. Let $\mathcal{T}$ be the stable lattice which satisfies the condition (F). We denote by $T=\mathcal{T} \otimes \varphi(\mathbb{I})$ and let $$\pi : \mathcal{T} \twoheadrightarrow \mathcal{T}\otimes\varphi(\mathbb{I})=T$$
be the reduction map. We have $T$ is a $\mathrm{mod}\ \varpi_{\varphi}$ not semi-simple by Lemma \ref{Maz}. Let $$\mathcal{L}\left(\rho_{f_{\varphi}} \right)=\Set{[T], [T_1], \cdots, [T_n]}$$ such that for any $1 \leq i \leq n$, $T/T_i \cong \varphi(\mathbb{I})/\left(\varpi_{\varphi} \right)^i$ as a $\varphi\left(\mathbb{I} \right)$-module. We denote by $\mathcal{T}_i=\pi^{-1}(T_i)$. Since $\mathcal{PT} \subset \mathcal{T}_i \subset \mathcal{T}$, $\mathcal{T}_i$ is a lattice. By the definition of $\mathcal{T}_i$ we have $\mathcal{T}_i$ is stable under the $\mathrm{Gal}\left(\overline{\mathbf{Q}}/\mathbf{Q} \right)$-action. Thus we obtain stable $\mathbb{I}$-lattices

$$\mathcal{T}\supset\mathcal{T}_1\supset\cdots\supset\mathcal{T}_n.$$
For $i \neq j$, if there exists an $\mathbb{I}[\mathrm{Gal}\left(\overline{\mathbf{Q}}/\mathbf{Q} \right)]$-isomorphism $\Xi : \mathcal{T}_i \stackrel{\sim}{\rightarrow} \mathcal{T}_j$, then $\Xi$ induces a $\varphi(\mathbb{I})[\mathrm{Gal}\left(\overline{\mathbf{Q}}/\mathbf{Q} \right)]$-isomorphism $$T_i \stackrel{\sim}{\rightarrow} T_j, v\otimes 1 \mapsto\Xi(v)\otimes 1$$ in $\mathcal{T}\otimes_{\mathbb{I}}\varphi(\mathbb{I})$. For $i \neq j$, $T_i$ and $T_j$ are non-isomorphic to each other by Proposition \ref{2.8}. This contradicts to our assumption. Hence $\mathcal{T}_i$ and $\mathcal{T}_j$ are non-isomorphic to each other and $\sharp\mathcal{L}(\rho_{\mathscr{F}}) \geq n+1 > m.$ This completes the proof of Corollary \ref{yandong3}. 

\end{pf3}

\begin{rem}
By Corollary \ref{m}, we also have $I\left(\rho_{\mathscr{F}} \right)=\hat{G}_{\chi_1^{-1}\chi_2}(X)\mathbb{I}$ is a prime ideal under the conditions (D), (R) and (P). Thus Corollary \ref{yandong3} is also satisfied if we assume the conditions (D), (R), (P) and (F). 
\end{rem}

\section{Examples}

Let $\mathcal{O}$ be the ring of integers of a finite extension of $\mathbf{Q}_p$ and $f \in S_k(\Gamma_0(M), \varepsilon, \mathcal{O})$ a newform. Assume that the eigenvalue $a(p, f)$ of $f$ for the Hecke operator $T_p$ is a $p$-adic unit. We define $f^{*} \in S_k(\Gamma_0(Mp), \varepsilon, \mathcal{O})$ by $f^{*}=f(q)-\beta f(q^p)$, where $\beta$ is the unique root of $x^2-a(p, f)x+\psi(p)p^{k-1}$ with $p$-adic absolute $|\beta| < 1$. We call this $f^{*}$ the $p$-stabilized newform associated to $f$. 

Let $(p, k_0)$ be the irregular pair such that $p\mid B_{k_0}$. We give two examples as follows: 

\begin{list}{}{}
\item[1.]$(p, k_0)=(691, 12)$. Let $\Delta \in S_{12}\left(\mathrm{SL}_2(\mathbf{Z}) \right)$ be the Ramanujan's cuspform. Since $\mathrm{dim}_{\mathbf{C}}S_{12}\left(\mathrm{SL}_2(\mathbf{Z}) \right)=1$, there exists an unique $\Lambda$-adic normalized Hecke eigen cusp form $\mathscr{F} \in S^{\mathrm{ord}}(\omega^{11}, \Lambda)$ such that $$\mathscr{F}\left(u^{10}-1 \right)=\Delta^{*}$$

and $\mathbf{T}(\omega^{11}, \Lambda)$ is isomorphic to $\Lambda$ (see \cite[\S 7.6]{H3}), where $\Lambda=\mathbf{Z}_p[[X]]$. Hence $I\left(\rho_{\mathscr{F}} \right)=\hat{G}_{\omega^{11}}(X)\mathbb{I}$ by Corollary \ref{m}. The ideal generated by the Iwasawa power series $\left(\hat{G}_{\omega^{11}}(X) \right)$ is equal to $\left(X-a_{\omega^{11}} \right)$ with $a_{\omega^{11}} \in p\mathbf{Z}_p \setminus p^2\mathbf{Z}_p$ which is calculated by Iwasawa-Sims (see \cite[\S1]{Wag}). Then we have the following statements:

\begin{list}{}{}
\item[(i)] $\sharp\mathcal{L}(\rho_{f_{\varphi}})$ is unbounded when $\varphi$ varies in $\mathfrak{X}_{\Lambda, 1}$ by (4) of Corollary \ref{yandong2}.
\item[(ii)] $\sharp\mathcal{L}(\rho_{f_{\varphi}})=2$ is constant when $\varphi$ varies in $\mathfrak{X}_{\Lambda}^{(0)}$ by (1) and (3) of Corollary \ref{yandong2}. 

\item[(iii)] For each $k \geq 2$, $\sharp\mathcal{L}(\rho_{f_{\varphi}})$ is bounded with maximum value $\mathrm{ord}_p(L_p(1-k, \chi\omega))+1$ when $\varphi$ varies in $\mathfrak{X}_{\Lambda, k}$ by (i) and (ii). 

\item[(iv)] Since $\mathbb{I}=\Lambda$ is a regular local ring, for a stable $\Lambda$-lattice $\mathcal{T}$, we have that $$\mathcal{T}^{**}=\mathrm{Hom}_{\Lambda}\left(\mathrm{Hom}_{\Lambda}\left(\mathcal{T}, \Lambda \right), \Lambda \right)$$ is a $\Lambda$-free lattice. Hence the condition (F) is satisfied and we have $\sharp\mathcal{L}(\rho_{\mathscr{F}})=\infty$ by Corollary \ref{yandong3}. 

\end{list}

\begin{rem}
\cite[Appendix II]{Maz} tells us that for the irregular pair $(p, k_0)$ with $p < 10^7$ and $k_0 < 8000$ such that $p\mid B_{k_0}$,  $\mathbf{T}(\chi, \Lambda)$ is isomorphic to $\Lambda$ except for $(p, k_0)=(547, 486)$. Hence we can apply Theorem \ref{yandong} (2) for these pairs. 
\end{rem}
\item[2.]$(p, k_0)=(547, 486)$. By \cite[Appendix II]{Maz}, there is a conjugate pair of $p$-stabilized newforms of weight 486 with the required Eisenstein congruence condition mod 547 and the corresponding Hida Hecke algebra $\mathbf{T}(\omega^{485}, \Lambda)$ is finite flat of rank two over $\Lambda$. We denote by $f_{486}^{*}, f_{486}^{\prime*}$ the corresponding newforms.

Let $\mathscr{F}$ (resp. $\mathscr{F}^{\prime}$) be the $\mathbb{I}$-adic normalized Hecke eigen cusp form associated to $f_{486}^{*}$ (resp. $f_{486}^{\prime*}$). Note that $\mathbb{I}$ is an integral closure of a quotient of $\mathbf{T}(\omega^{484}, \Lambda)$ by a minimum prime ideal of $\Lambda$ by the proof of \cite[\S 7.4, Theorem 7]{H3}. Hence $\mathrm{Frac}(\mathbb{I})$ is a quadratic extension of $\mathrm{Frac}(\Lambda)$. The ideal generated by the Iwasawa power series $\left(\hat{G}_{\omega^{485}}(X) \right)$ is equal to $\left(X-a_{\omega^{485}} \right)$ with $a_{\omega^{485}} \in p\mathbf{Z}_p \setminus p^2\mathbf{Z}_p$ which is calculated by Iwasawa-Sims (see also \cite[\S1]{Wag}). Then $\sharp\mathcal{L}(\rho_{f_{\varphi}}) \in \{2, 3\}$ when $\varphi$ varies in $\mathfrak{X}_{\mathbb{I}}^{(0)}$ by (1) of Corollary \ref{yandong2}. The condition (C) satisfied for $\mathscr{F}$ (this is because the Vandiver's conjecture is true for $p=547$), thus $I(\rho_{\mathscr{F}})$ is a principal ideal which is generated by a factor of $X-a_{\omega^{485}}$ in $\mathbb{I}$ by Corollary \ref{r}.
The same holds for $\mathscr{F}^{\prime}$.

\end{list}

\end{document}